\documentclass[12pt]{article}
\usepackage[margin=1in]{geometry}
\usepackage{authblk}

\usepackage{qstyle}

\title{Quantum expanders and property (T) \\ discrete quantum groups}

\author[$\dagger$]{Michael Brannan}
\author[$\star$]{Eric Culf}
\author[$\ddag$]{Matthijs Vernooij}

\affil[$\dagger$]{\small{Department of Pure Mathematics, University of Waterloo, Canada, \texttt{michael.brannan@uwaterloo.ca}}}
\affil[$\star$]{Department of Applied Mathematics and Institute for Quantum Computing, University of Waterloo, Canada, \texttt{eculf@uwaterloo.ca}}
\affil[$\ddag$]{Delft Institute of Applied Mathematics, TU Delft, The Netherlands, \texttt{m.n.a.vernooij@tudelft.nl}}

\begin{document}

\maketitle

\begin{abstract}
    Families of expander graphs were first constructed by Margulis from discrete groups with property (T). Within the framework of quantum information theory, several authors have generalised the notion of an expander graph to the setting of quantum channels. In this work, we use discrete quantum groups with property (T) to construct quantum expanders in two ways.  The first approach obtains a quantum expander family by constructing the requisite quantum channels directly from finite-dimensional irreducible unitary representations, extending earlier work of Harrow using groups. The second approach directly generalises Margulis' original construction and is based on a quantum analogue of a Schreier graph using the theory of coideals. To obtain examples of quantum expanders, we apply our machinery to discrete quantum groups with property (T) coming from compact bicrossed products.    
\end{abstract}

\section{Introduction}
Expander graphs (cf. Section \ref{subsec:c-expanders}) are particularly well-connected sparse graphs. They appear in many applications, such as robustly connected networks, error correcting codes, the complexity of local Hamiltonians, pseudorandomness, the Baum-Connes conjecture and Monte-Carlo simulations \cite{HLW06,AE15}. In the field of quantum information theory, the notion of quantum expanders (cf. Section \ref{subsec:q-expanders}) arises naturally as an extension of classical expander graphs \cite{Has07a,Har07}. Just as in the classical case, randomness, in this case in the form of random matrix techniques, provides an important way to construct quantum expanders \cite{Has07b, LaYo23}.

To give explicit examples of expander families, deterministic constructions are needed. These constructions tend to be based on the representation theory of groups, for example Margulis's original idea to construct classical expanders using property (T) groups (cf. Section \ref{subsubsec:prop-T-expanders}) and Harrow's construction of quantum expanders using finite groups (cf. Section \ref{subsubsec:q-expander-from-c-expander}) \cite{Mar73,Har07}.  Both of these approaches use the fact that the geometry of the representation theory of a group $\Gamma$ (e.g., finiteness, amenability, property (T), and so on) provides information about the spectral gap of certain convolution operators on $\Gamma$ (or its homogeneous spaces), and this in turn provides (via Cheeger inequalities) bounds on the expansion constants of related graphs or quantum channels.

This work aims to develop the representation-theoretic tools for constructing general quantum expanders. Discrete quantum groups and their representations can be used to construct more general examples of bistochastic quantum channels (cf. Section \ref{subsec:q-channels-from-q-groups}), which, in contrast to quantum channels coming from groups, are often not mixed unitary. Since property (T) makes sense for discrete quantum groups (cf Section \ref{subsubsec:dqg-duality-prop-T}), it is natural to ask if Margulis' construction is possible in the quantum setting. We show that this is indeed the case.

More precisely, we are able to construct quantum expanders from discrete quantum groups with property (T) in two ways. First, we extend the construction of quantum expanders by Harrow to obtain the following result (see Theorem \ref{thm:quantum-margulis} for the precise formulation):
\begin{theoremA} \label{thm:A}
    Let $\dqG$ be an infinite property (T) discrete quantum group with fixed symmetric Kazhdan pair $(E,\epsilon) \in \mathrm{Irred}(\cqG) \times \mathbb R_{>0}$, where $\mathrm{Irred}(\cqG)$ denotes the set of equivalence classes of irreducible unitary representations of the compact quantum group $\cqG$ dual to $\dqG$.  Let $(U_i)_{i\in\mathbb{N}}$ be a sequence of finite-dimensional irreducible unitary representations of $\dqG$ in $\M_{n_i}$ with $\lim_{i\to \infty}\ n_i =\infty$. Then one can construct a family of bistochastic quantum channels $\mathcal F = (\Psi_i:\M_{n_i} \to \M_{n_i})_{i\in \mathbb{N}}$ which is a bounded degree quantum expander family with quantum edge expansion 
    \begin{equation*}
       h_Q(\mathcal F):= \inf_{i\in \mathbb{N}}h_Q(\Psi_i)\geq K\epsilon^2, 
    \end{equation*}
    where $K = (4\sum_{x \in E} \dim x)^{-1}$.
\end{theoremA}

When $\dqG$ is a classical discrete group, we have $\dim x = 1$ for all $x \in \text{Irred}(\cqG)$, giving $K = \frac{1}{4|E|}$, which is the familiar constant appearing in previous constructions of expander graphs using property (T) groups.  When $\dqG$ is a genuine discrete quantum group (i.e., not actually a group), the resulting family of expanders $\mathcal F$  appearing above has the interesting feature that it may no longer be of mixed unitary type.  Thus Theorem A can be seen as means to construct explicit quantum expanders of a rather different flavour than those currently existing in the literature.  We use Theorem A to construct and examine quantum expanders arising from compact bicrossed product quantum groups in Section \ref{sec:bicrossed}.

In recent years, a theory of (finite) {\it quantum graphs} has emerged, based on the notion of a {\it quantum adjacency matrix}  \cite{MRV18, Was23, Daw24}.  The fundamental idea is to consider a non-commutative vertex set modelled by a finite dimensional von Neumann algebra $M$, equipped with a completely positive linear map $A:M \to M$ satisfying a certain Schur-idempotency condition.  The pair $(M,A)$ is then called a quantum graph.   We review these ideas in Section \ref{sec:qam}.  From the quantum adjacency matrix view on quantum graphs, notions like regularity, connectedness, and spectral gap all have natural generalizations.  The study of quantum expanders from the perspective of quantum adjacency matrices has yet to be studied in any detail in the literature, and we provide some first contributions in this direction.  First, we show that the commonly used notions of degree ( = the Kraus rank) for bistochastic quantum channels on $\M_n$ and for quantum adjacency matrices on $\M_n$ agree (Proposition \ref{prop:kr=degree}).  Next, using Wasilewski's recent formulation of quantum Cayley graphs using quantum adjacency matrices \cite{Was23} and the notion of a coideal in a discrete quantum group, we  formulate a natural notion of a \textit{quantum Schreier graph}.   We then use these ideas to prove Theorem \ref{thm:spectral-gap-schreier}, which can be regarded as a more direct generalization of Margulis' original construction of expander graphs as Schreier coset graphs:
\begin{theoremB}
    Let $\dqG$ be an infinite property (T) discrete quantum group with symmetric Kazhdan pair $(E,\epsilon)$. Denote by $A: \ell^\infty(\dqG) \to \ell^\infty(\dqG)$ the convolution operator $Ax = p_E \star x$, ($x \in \ell^\infty(\dqG)$), describing the quantum Cayley graph $\mathcal C(\cqG, E)$, where 
 $p_E$ is the central projection associated to $E$.  For  each $i \in \mathbb N$, let $M_i \subseteq \ell^\infty(\dqG)$ be a  finite-dimensional coideal, and assume moreover $\lim_{i \to \infty}\dim(M_i)=\infty$.  Then the family  $(M_i,A|_{M_i})$ of {\it quantum Schreier graphs } forms a quantum expander family of bounded degree.
\end{theoremB}

As we explain in Section \ref{sec: Schreier}, when $\dqG$ is classical, our quantum Schreier graphs are in  one-to-one correspondence with the finite Schreier coset graphs of $\cqG$, recovering Margulis' result.  In quantum case, this will generally yield 
quantum expander graphs over multimatrix algebras.  

The remainder of the paper is organized as follows: Section \ref{background} contains a detailed discussion of classical expander graphs, Cheeger inequalities, and their quantisations.  We 
 also provide a detailed recap of the relevant facts about quantum groups, their associated quantum channels, and property (T) in Section \ref{sec:qg}.  Section \ref{sec:edge-expanders} is dedicated to proving Theorem A, while Section \ref{sec:bicrossed} applies these results to examples coming from bicrossed products.  Section \ref{sec: Schreier} is then dedicated to introducing the notion of a quantum Schreier graph, and proving Theorem B.  Finally, we conclude the paper with a brief discussion and outlook for future work in Section \ref{sec:discussion}.

\section{Background}\label{background}

\subsection{Classical expanders} \label{subsec:c-expanders}
In this section, we recall 
the definition of a (classical) expander graph, and outline the construction of an expander family from a property (T) group due to Margulis~\cite{Mar73}. Our main reference for this section is the book of Kowalski~\cite{Kow19}.

\subsubsection{Expander graphs}

\begin{definition}
    A (simple undirected loop-free) \textit{graph} is a pair of sets $G=(V,E)$, where $V$ is the set of vertices and $E$ is the set of edges, corresponding to pairs of distinct elements from $V$. We assume $V$ is finite unless otherwise specified. For a graph $G$, we use the following notation:
    \begin{itemize}
        \item The \textit{neighbourhood} of a vertex $v\in V$ is $N(v)=\set*{u\in V}{\{u,v\}\in E}$.
        \item We say $G$ is \textit{$d$-regular} if $\abs*{N(v)}=d$ for all $v\in V$.
        \item Given $X,Y\subseteq V$, write $E(X,Y)=\set*{\{x,y\}\in E}{x\in X,\,y\in Y}$ for the set of edges between $X$ and $Y$.
        \item The \textit{adjacency matrix}, written $A(G)$, is the $|V|\times|V|$ matrix with entries labelled by pairs of vertices such that $A(G)_{x,y}=1$ iff $\{x,y\}\in E$ and $0$ otherwise.
    \end{itemize}
\end{definition}

\begin{definition}
    The \textit{expansion constant} (or Cheeger constant) of a graph $G$ is
    $$h(G)=\min\set*{\frac{|E(U,V \backslash U)|}{\min\{|U|,|V \backslash U|\}}}{\varnothing\subsetneq U\subsetneq V}.$$
\end{definition}

Note that we have to restrict to connected graphs for this to be non-zero: the expansion constant of any disconnected graph is $0$.

The expansion constant is closely related to the spectral gap of the adjacency matrix of the graph. This relationship gives an important tool for studying expansion properties. First, note that if $G$ is $d$-regular, then the largest eigenvalue of $A(G)$ is $d$, and if $G$ is connected, the multiplicity of this eigenvalue is $1$. Denote the second-largest eigenvalue of $A(G)$ by $\lambda_2(G)$. Then, the spectral gap is $d-\lambda_2(G)$, which relates to $h(G)$ as follows.

\begin{theorem}[Discrete Cheeger inequalities~\cite{Dod84,AM85}]\label{thm:classical-cheeger}
    Let $G$ be a connected $d$-regular graph. Then,
    $$\frac{1}{2}(d-\lambda_2(G))\leq h(G)\leq\sqrt{2d(d-\lambda_2(G))}.$$
\end{theorem}

There is a slightly stronger form of the upper bound, given by $\sqrt{d^2-\lambda_2(G)^2}$ due to~\cite{Moh89}. See~\cite{Kow19} for extensions of these inequalities to non-regular graphs.

\begin{definition}
Let $\eta >0$ and $d \in \mathbb N$.  We say a graph $G$ is a \textit{$(d,\eta)$-expander} if it is $d$-regular and $h(G)\geq \eta$.    A \textit{$(d,\eta)$-expander family} is a sequence $(G_n)$ of $d$-regular graphs such that $|G_n|\rightarrow\infty$ and $h(G_n)\geq\eta$.
\end{definition}

An important use of expanders is in \textit{expander random walk sampling}. Informally, a random walk on an expander graph reaches any vertex in the graph very quickly. One application of this is to reduce the amount of randomness needed to take an average of a function over the vertices, by taking the average over a short random walk~\cite{AKS87}. See~\cite{Gil98} for the formal statement and a full proof. Importantly for our context, in the random walk picture, a $d$-regular graph $G$ gives rise to a bistochastic matrix $\frac{1}{d}A(G)$, which is the description which quantises most naturally.

\subsubsection{Property (T) groups and expanders} \label{subsubsec:prop-T-expanders}

A natural way to construct examples of regular graphs is from groups.

\begin{definition}
    Let $\Gamma$ be a finitely-generated group with a generating set $S$ that is symmetric, \textit{i.e.} $S^{-1}=S$. The \textit{Cayley graph} of $\Gamma$ with respect to $S$ is the graph $\mc{C}(\Gamma,S) = (\Gamma,E)$ with $E = \set*{\{g,gs\}}{g\in\Gamma,\;s\in S}$.
\end{definition}

Note that $\mc{C}(\Gamma,S)$ is finite if $\Gamma$ is finite and is infinite for infinite finitely-generated groups. Also, the Cayley graph can be defined with respect to an arbitrary set $S$, but if it is not a generating set, $\mc{C}(\Gamma,S)$ will not be connected. As will be noted below for quotient spaces, the notion of a Cayley graph can be extended to sets on which $\Gamma$ acts.

The Margulis construction of expanders \cite{Mar73, Lub94} relies on a group with Kazhdan's property (T). We give only the definition for discrete groups here, as this is relevant to our situation.

\begin{definition}
    Let $\Gamma$ be a discrete group and $\pi:\Gamma\rightarrow B(\mathcal H)$ be a unitary representation. For a finite set $S\subseteq\Gamma$ and $\epsilon>0$, we say that $v\in \mathcal H\backslash\{0\}$ is an \textit{$(S,\epsilon)$-invariant vector} for $\pi$ if $\norm{\pi(g)v-v}<\epsilon\norm{v}$ for all $g\in S$.

    We say $\Gamma$ has \textit{property (T)} if for every unitary representation $\pi$ on a Hilbert space $\mathcal H$ such that there exists an $(S,\epsilon)$-invariant vector for all finite sets $S\subseteq\Gamma$ and $\epsilon>0$, there exists an invariant vector $v\in \mathcal H\backslash\{0\}$, that is $\pi(g)v=v$ for all $g\in\Gamma$. 
\end{definition}

Property (T) is equivalent to the existence of a \textit{Kazhdan pair}: a fixed pair $(S,\epsilon)$ with $S$ a finite symmetric generating set, and $\epsilon >0$ such that if a unitary representation has an $(S,\epsilon)$-invariant vector, then it has an invariant vector.

We now state the main result of Margulis. 

\begin{theorem}[Margulis \cite{Mar73}]
    Let $\Gamma$ be a discrete group with property (T), and let $(S,\epsilon)$ be a Kazhdan pair for $\Gamma$ such that $S$ is a generating set. For any  subgroup of finite index $H\le\Gamma$,
    $$h(\mc{C}(\Gamma/H,S))\geq\frac{\epsilon^2}{4}.$$
\end{theorem}

In the above theorem, $\mc{C}(\Gamma/H,S)$ denotes the induced {\it Schreier coset graph} on the quotient space $\Gamma/H$.  The vertices are given by the cosets $\{gH: g \in \Gamma\} = \Gamma/H$, and the  edges are given by 
\[
E = \set{\{gH, sgH\}}{ s \in S}.
\]
Hence, this construction gives rise to an $(|S|,4^{-1}\epsilon^2)$-expander, and if $\Gamma$ has a sequence of subgroups with arbitrarily large finite index, then the construction gives an $(|S|,4^{-1}\epsilon^2)$-expander family.

\begin{example}
    The group $\Gamma=\mathrm{SL}_3(\Z)$ has property (T), and admits a Kazhdan pair $(S,\epsilon)$, where $S=\set*{I\pm e_{ij}}{i\ne j}$, for $e_{ij}$ the canonical matrix units, and $\epsilon=\frac{1}{42\sqrt{3}+860}>0.001$~\cite{Sha99,Kas05}. Also, by taking the quotients $\mathrm{SL}_3(\Z)\rightarrow\mathrm{SL}_3(\Z_p)$ for $p$ prime we get an infinite sequence of quotient groups of strictly increasing order. Thus, by the theorem above, we get a $(12,0.25 \times 10^{-6})$-expander family.
\end{example}

In Section \ref{sec: Schreier}
 we extend Margulis' theorem to the setting of discrete quantum groups (cf. Theorem \ref{thm:spectral-gap-schreier}).
 
\subsection{Quantum expanders} \label{subsec:q-expanders}

In this section, we define what we mean by a quantum expander and recall the construction of Harrow~\cite{Har07} that allows us to construct a quantum expander from a classical expander constructed from a finite group, as in the previous section. In Section \ref{sec:edge-expanders} we generalise this class of quantum expanders to those arising from discrete quantum groups.

\subsubsection{Quantum bistochastic maps}

The notion of a quantum expander was first introduced in~\cite{Has07a}. There are two different ways that quantum expanders are represented in the literature: they can be represented as tuples of operators, as in~\cite{Has07a,Has07b,LQWWZ22arxiv}, or they can be represented as quantum channels, as in~\cite{Har07}. We use the latter definition, as all of the spectral and expansion properties are most readily phrased in this setting, and it naturally generalises the presentation of an expander graph as a bistochastic matrix.

\begin{definition}  Let $\mathcal K$ be a finite-dimensional Hilbert space.  We call a linear map ${\Phi:B(\mathcal K) \to B(\mathcal K)}$ a \textit{quantum bistochastic map} if it is completely positive, trace-preserving, and unital. We say $\Phi$ is \textit{undirected} if it is hermitian with respect to the Hilbert-Schmidt inner product. We say $\Phi$ is \textit{connected} if its $1$-eigenspace is one-dimensional. 
\end{definition}

\begin{example}[\cite{LQWWZ22arxiv}]\label{ex:graph-to-channel}
    Given a $d$-regular graph $G=(V,E)$, there is a natural way to construct a quantum bistochastic map from it. First, we construct a family of disjoint vertex cycle covers for $G$ inductively. Pick an arbitrary vertex cycle cover $G_1$ of $G$, that is a $2$-regular subgraph that contains all the vertices of $G$.  It must be composed of a disjoint union of cycles. Now, remove the edges of $G_1$ from $G$, giving a $d-2$-regular graph. We repeat this process $\floor{d/2}$ times to get disjoint cycle covers $G_1,\ldots,G_{\floor{d/2}}$. If $d$ is even, there are no remaining edges. If $d$ is odd, the remaining graph is $1$-regular, consisting of a disjoint union of pairs of vertices with one edge between them. Let this graph be $G'$. For each $i$, we can write the adjacency matrix $A(G_i)=P_{2i-1}+P_{2i}$, where $P_{2i-1}$ is a permutation matrix corresponding to traversing the cycles in one direction, and $P_{2i}=P_{2i-1}^\ast$ is the permutation matrix corresponding to travelling in the other direction. In the case that $d$ is odd, let $P_d=A(G')$: this is a hermitian permutation matrix. Now, we can define the map $\Phi_G:B(\C^V)\rightarrow B(\C^V)$ by
    $$\Phi_G(\rho)=\frac{1}{d}\sum_{i=1}^dP_i\rho P_i^\ast$$
    Note that this map is not uniquely determined by $G$ but depends on the choice of permutations $P_i$, which are not uniquely determined by the above procedure. By construction, $\Phi_G$ is a completely positive linear map, and as the permutation matrices are unitary, this is also a mixed-unitary channel, and hence both trace-preserving and unital. Thus, $\Phi_G$ is a quantum bistochastic map. 

    Also, we see that the action of $\Phi_G$ on the diagonal subspace $D=\spn\set*{\ketbra{x}}{x\in V}\cong\C^V$ is exactly the action of the normalised adjacency matrix $\frac{1}{d}A(G)$ on $\C^V$ (where $\mathbb C^V$ is identified with the diagonal subalgebra of $B(\mathbb C^V)$). In fact,
    $$\Phi_G(\ketbra{x})=\frac{1}{d}\sum_{i=1}^dP_i\ketbra{x}P_i^\ast=\frac{1}{d}\sum_{y\in N(x)}\ketbra{y}.$$
    Next, we see that $\Phi_G$ is undirected. For $i=1,\ldots,\floor{d/2}$, $P_{2i-1}^\ast=P_{2i}$ and if $d$ is odd, $P_d^\ast=P_d$, so $\Phi_G$ is hermitian.
    
    However, note that $\Phi_G$ may not be connected even if $G$ is. For example, consider the graph $G$ with two vertices and one edge connecting them. Then, the adjacency matrix $A(G)=\begin{bmatrix}0&1\\1&0\end{bmatrix}$. As such, $\Phi_G(\rho)=A(G)\rho A(G)$, whose $1$-eigenspace is $\spn\{I,A(G)\}$ and hence has dimension two. Nevertheless, as will be seen in the next subsection, if we see the $P_i$ as corresponding to the generators of a group in the left-regular representation, we can restrict to a block corresponding to an irreducible representation and hence recover connectedess (and expansion).
\end{example}

We now come to a natural notion of expansion constant for quantum bistochastic maps introduced in \cite{Has07b}.  See also \cite{LQWWZ22arxiv}.

\begin{definition} \label{def:quantumexpansion}
    The \textit{quantum (edge) expansion} of a quantum bistochastic map $\Phi:B(\mc{K}) \to B(\mc{K})$ is
    $$h_Q(\Phi)=\min\set*{\frac{\Tr\squ*{(I-\Pi)\Phi(\Pi)}}{\min\{\Tr(\Pi),\Tr(I-\Pi)\}}}{\Pi\in B(\mc{K})\text{ projector},\;\Pi\neq 0,I}.$$
\end{definition}

\begin{remark}
One can make sense of the above definition of expansion in the  setting of a finite von Neumann algebra $M$ equipped with a faithful normal tracial state $\tau:M \to \mathbb C$, and a $\tau$-preserving completely positive map $\Phi:M \to M$.  In this case, $\tau$ replaces the usual trace $\Tr$ in Definition \ref{def:quantumexpansion}.  The advantage of this more general setup is that it also captures the classical notion of expansion for $d$-regular graphs $G = (V,E)$, where one takes $M = C(V)$ (the algebra of functions on $V$) and $\tau$ the uniform probability on $V$.  See the discussion in \cite{Has07b} for more details.     
\end{remark}

Note that, as for a classical graph, there is nontrivial expansion only if the map is connected, according to our definition of connectedness.

\begin{lemma}
    If a quantum bistochastic map $\Phi$ is not connected, $h_Q(\Phi)=0$.
\end{lemma}

\begin{proof}
    First, note that $I$ is an eigenvector of $\Phi$ with eigenvalue $1$. If the dimension of the $1$-eigenspace is not one, there exists a hermitian matrix $X$ such that $\Phi(X)=X$, as $\Phi$ is hermitian-preserving. Thus, taking a linear combination of $I$ and $X$, there is a positive matrix $P$ that is not full rank such that $\Phi(P)=P$. Let $\Pi$ be the projector onto the support of $P$. Then there exist $\lambda,\mu>0$ such that $\lambda\Pi\leq P\leq\mu\Pi$. As such, $\Phi(\Pi)\leq\frac{1}{\lambda}\Phi(P)=\frac{1}{\lambda}P\leq\frac{\mu}{\lambda}\Pi$, giving that $\Tr\squ*{(I-\Pi)\Phi(\Pi)}\leq\frac{\mu}{\lambda}\Tr\squ*{(I-\Pi)\Pi}=0$ and hence $h_Q(\Phi)=0$.
\end{proof}

The notion of expansion above gives rise to a natural analogue of the Cheeger inequalities. As seen classically in the previous section, there is a relationship between the second-largest eigenvalue $\lambda_2(\Phi)$ and the quantum expansion.

\begin{theorem}[\cite{Has07b}]\label{thm:quantum-cheeger}
    Let $\Phi:B(\mc{H})\rightarrow B(\mc{H})$ be a connected undirected quantum bistochastic map. Then,
    $$\frac{1}{2}(1-\lambda_2(\Phi))\leq h_Q(\Phi)\leq\sqrt{2(1-\lambda_2(\Phi))}.$$
\end{theorem}

Note that the statement of the above theorem differs by a factor of $d$ from~\cref{thm:classical-cheeger}, the case of classical graphs. This is because $h_Q(\Phi)$ is defined for the quantum analogue of the bistochastic matrix $\frac{1}{d}A(G)$, and not for the adjacency matrix $A(G)$.  

In order to properly speak about expanders in the quantum setting, we must first have a good notion of degree for a quantum bistochastic map. 
Let $\mc{K}$ be a finite dimensional Hilbert space and $\Phi: B(\mc{K}) \to B(\mc{K})$ be a quantum channel.  Let $\{K_i\}_{i=1}^d \subseteq B(\mc{K})$ be a family of Kraus operators for $\Phi$.  That is, $\Phi$ can be represented as 
\[
\Phi(\rho) = \sum_{i=1}^d K_i\rho K_i^*. 
\]
The minimal size $1 \le d \le (\dim \mc{K})^2$ of a family of Kraus operators representing $\Phi$ is called the {\it Kraus rank} of $\Phi$.  Note that the Kraus rank of $\Phi$ is the rank of the Choi matrix $J_\Phi = \sum_{ij} e_{ij} \otimes \Phi(e_{ij}) \in B(\mathcal K) \otimes B(\mathcal K)$.  Following \cite{Har07, LaYo23, LQWWZ22arxiv}, we make the following definition.

\begin{definition} \label{def:quantum-degree}
The {\it degree} of quantum bistochastic map $\Phi:B(\mc{K}) \to B(\mc{K})$, $\deg \Phi$, is defined to be the Kraus rank of $\Phi$.     
\end{definition}

Note that for classical $d$-regular graphs $G$, the somewhat complicated extension of $\frac{1}{d}A(G)$ to a quantum bistochastic map $\Phi_G$ in Example \ref{ex:graph-to-channel} has the nice property that it is degree-preserving.  On the other hand, the ``canonical'' extension of $\frac{1}{d}A(G)$ to a different quantum bistochastic map $\tilde \Phi_G$ by using the Kraus operators $\set{d^{-\sfrac{1}{2}}e_{ij}}{\{i,j\} \in E}$ is not degree-preserving.

With the notion of degree in hand, we can finally define quantum expanders (also called quantum edge expanders). 

\begin{definition}\label{def:qexpander}
   Let $\eta > 0$ and $d \in \mathbb N$. We say a quantum bistochastic map $\Phi$ is a \textit{$(d,\eta)$-expander} if $\deg \Phi = d$ and $h_Q(\Phi)\geq \eta$.    
   A (constant degree $d$) \textit{$(d,\eta)$-expander family} is a sequence $\Phi_n:B(\mc{K}_n)\rightarrow B(\mc{K}_n)$ of$(d,\eta)$-expanders such that $\dim \mathcal K_n\rightarrow\infty$.
\end{definition}

\subsubsection{Quantum adjacency matrices} \label{sec:qam}

The above definition of degree for quantum bistochastic maps may seem somewhat unnatural, given that it assigns a global degree $d$ to {\it every} quantum bistochastic map $\Phi$ (even those coming from non-regular graphs).  Instead, one might try to restrict to a special class of quantum bistochastic maps $\Phi$ which have the property that there exists $d >0$ such that $d\Phi$ is some sort of quantum analogue of a ``$\{0,1\}$-matrix''.  This idea has been formalized in \cite{MRV18}, and leads to the notion of a {\it quantum graph} and a {\it quantum adjacency matrix}.  See also \cite{Daw24, Was23, CW22} for a treatment that is closer to ours below.  In the following definition, we restrict to what are called {\it tracial} quantum graphs, as they are what will appear in all our examples.

\begin{definition}{\cite{MRV18}}\label{def:qadjacency}
Let $M = \bigoplus_{a}\M_{n_a}$ be a finite-dimensional von Neumann algebra, let $\psi:M\to \mathbb C$ be the faithful trace given by $\psi = \sum_{a} n_a\Tr_{\M_{n_a}}(\cdot)$, and let $m:M \otimes M \to M$ denote the multiplication map.  A completely positive map $A:M \to M$ is called a {\it quantum adjacency matrix} if it satisfies 
\[
m(A \otimes A)m^* = A,
\]
where $m^*:M \to M \otimes M$ denotes the Hilbert space adjoint of $m$ induced by the Hilbert space structures coming from the traces $\psi$ and $\psi \otimes \psi$.  The pair $G = (M,A)$ is called a {\it quantum graph}. We say that a quantum graph $G = (M,A)$ is {\it undirected} if $A$ is hermitian with respect to the canonical Hilbert space structure on $M$ induced by $\psi$.   
\end{definition}

When $M$ is abelian, the above notions reduce to the usual ones for graphs and adjacency matrices.  See the discussion in  \cite{MRV18, Daw24} for example

For quantum graphs $G = (M,A)$, the notion of regularity is quite natural to define.

\begin{definition} \label{def:regularqgraph}
An undirected quantum graph $G = (M,A)$ is {\it $d$-regular} if $A1 = A^*1 =  d1 $ for some $d > 0$.    
\end{definition}

\begin{remark}
In \cite{Was23} the definition of quantum graph is extended to the setting of infinite-dimensional $M$ being given by infinite direct products of matrix algebras.  Here the technical issue is that $m^*$ is unbounded, so care must be taken in interpreting the definition of a quantum adjacency matrix.  The theory of infinite quantum graphs is important and quite natural as it naturally models the notion of a Cayley graph for an infinite discrete quantum group \cite{Was23}.  This notion of quantum Cayley graph will be introduced and used in Section \ref{sec: Schreier}.      
\end{remark}

\begin{example}\label{ex:graph-to-channel-is-qgraph}
Consider the map $\Phi_G$ constructed in~\cref{ex:graph-to-channel} from a $d$-regular graph.  Set $A = d\Phi_G$.  Then the pair $(B(\C^V), A)$ is a $d$-regular quantum graph. Indeed, using 
$m^\ast(\ketbra{u}{v})=|V|^{-1}\sum_{w\in V}\ketbra{u}{w}\otimes\ketbra{w}{v}$, we have 
\begin{align*}
    m\circ(A\otimes A)\circ m^\ast(\ketbra{u}{v})&=|V|^{-1}\sum_{w,i,j}P_i\ketbra{u}{w}P_i^\ast P_j\ketbra{w}{v}P_j^\ast\\
    &=|V|^{-1}\sum_{i,j}\Tr(P_i^\ast P_j)P_i\ketbra{u}{v}P_j^\ast\\
    &=d\Phi_G(\ketbra{u}{v}) \\
    &= A(\ketbra{u}{v}),
\end{align*}
as $\Tr(P_i^\ast P_j)=|V|\delta_{i,j}$ since the $\{P_i\}_i$ are disjoint permutations.    
\end{example}

Note that if $G = (M,A)$ is a $d$-regular quantum graph, then $\Phi = \frac{1}{d}A$ will define a quantum bistochastic map on $M$.  When $M = B(\mc{K})$ is a full matrix algebra, one can ask whether the two notions of degree we have introduced agree. Fortunately they do, and this is the next result.

\begin{proposition} \label{prop:kr=degree}
    A completely positive map $A:B(\mc{K}) \to B(\mc{K})$ is a quantum adjacency matrix if and only if the normalised Choi matrix
    \[P_A:=  \frac{1}{\dim \mc{K}}\sum_{ij} e_{ij} \otimes A(e_{ij})\] is a projection.  If $(B(\mc K), A)$ is moreover a $d$-regular quantum graph, then 
    \[d = \text{rank}(P_A) = \text{the Kraus rank of $A$}.\]
\end{proposition}

\begin{proof}
The first claim is a well-known fact about quantum graphs on full matrix algebras.  See for example \cite[Lemma 1.6 and Proposition 1.7]{CW22}.  The second claim is just a computation.  Indeed, if $A1 = d1$, then
\begin{align*}
\text{Kraus rank of $A$} &= \text{rank}(P_A)\\
&= (\Tr \otimes \Tr)(P_A) \\
&= \frac{1}{\dim \mc K}\sum_i \Tr(A(e_{ii}))\\
&= d \qedhere
\end{align*}
 \end{proof}

 In particular, for regular quantum graphs $(B(\mathcal K), A)$ over full matrix algebras, the degree $d$ is an integer between $1$ and $(\dim \mathcal K)^2$, and it is easy to see that every value of $d$ in this range is attained, for example by taking $\Phi$ to be the mixed-unitary channel corresponding to an equal mixture of $d$ orthogonal unitaries. 

It is important to note that  for general  quantum graphs, $G = (M,A)$ there does not seem to exists a definition of the expansion constant $h_Q(G)$ (although a natural one is implicit in our discussion above - see the Remark following Definition \ref{def:quantumexpansion}.)  More importantly, a version of the Cheeger inequality Theorem \ref{thm:quantum-cheeger} beyond the case of matricial quantum graphs (i.e., those with $M = B(\mc K)$ - which are already covered by Theorem \ref{thm:quantum-cheeger}) does not exist.  A general version of the Cheeger inequality for (even non-tracial) quantum graphs has recently been announced in forthcoming work of Junk \cite{Ju24}.  
In any case, even without expansion constants and Cheeger inequalities, we can still talk about spectral gap for the bistochastic map $\Phi = d^{-1}A$ associated to a $d$-regular quantum $G = (M,A)$, and use this spectral gap to define $(d,\eta)$-expanders in this case.  This is what we shall do in Section \ref{sec: Schreier}, where we consider quantum analogues of Schreier graphs associated to quantum Cayley graphs of discrete quantum groups.

\subsubsection{Quantum expanders from classical expanders} \label{subsubsec:q-expander-from-c-expander}

In this final section, we recall the construction of a quantum expander based on a finite group due to Harrow~\cite{Har07}.

\begin{proposition}[\cite{Har07}]
    Let $\Gamma$ be a finite group with symmetric generating set $S$, and let $\pi:\Gamma\rightarrow\mc{U}(\mc H)$ be a nontrivial irreducible unitary representation. Then, the quantum channel
    $$\Phi(\rho)=\frac{1}{|S|}\sum_{g\in S}\pi(g)\rho\pi(g)^\ast$$
    is a connected undirected quantum bistochastic map such that $\lambda_2(\Phi)\leq\frac{1}{|S|}\lambda_2(\mc{C}(\Gamma,S))$.
\end{proposition}

We outline the idea of the proof below, as it contains many of the key conceptual ideas that will be used throughout the rest of the paper.  Note that the adjacency matrix of the Cayley graph $\mc{C}(\Gamma,S)$, when viewed as an operator on the Hilbert space $\ell^2(\Gamma)$, is given by $\sum_{g\in S}\lambda(g)$, where $\lambda$ is the left-regular representation. The left-regular representation decomposes as a direct sum of all irreducible representations with multiplicity given by the dimension. Hence the trivial representation $g\mapsto 1$ corresponds to the one-dimensional $1$-eigenspace of the normalised adjacency matrix, and for any non-trivial unitary irreducible representation $\rho$, the largest eigenvalue of $\sum_{g\in S}\rho(g)$ is upper-bounded by $\lambda_2(\mc{C}(\Gamma,S))$. Then, the matrix representation of $\Phi$ (viewed as a Hilbert space operator on $B(\mc H) \cong \mc H \otimes \bar{\mc H}$),  is given by $\frac{1}{|S|}\sum_{g\in S}\pi(g)\otimes\overline{\pi(g)}$.  Representation  theory then tells is that this operator can be decomposed as a direct sum of blocks of the form $\sum_{g\in S}\rho(g)$ for irreducible $\rho$, and there is only a single one-dimensional block in this decomposition corresponding to the trivial representation. Hence, we get the bound on the second-largest eigenvalue.

\subsection{Quantum groups} \label{sec:qg}
This section recalls some well known facts about quantum groups based on notes by Maes and Van Daele \cite{MvD98}, adjusting the notation to be more in line with recent papers \cite{Fim10, Was23}.  See also the book \cite{NT} for many of the results stated below without proof. We will focus on compact quantum groups, discrete quantum groups, the duality between them and their representations, and end the section with the definition of property (T) for discrete quantum groups due to Fima~\cite{Fim10}.

\subsubsection{Compact quantum groups}
\begin{definition}
    A \textit{compact quantum group} $\cqG$ is a pair $(A, \Delta)$ of a unital C*-algebra $A$ and a $\ast$-homomorphism $\Delta: A\rightarrow A\otimes_{\min} A$, called the \textit{comultiplication}, satisfying
    \begin{enumerate}[i.]
        \item $(\id\otimes \Delta) \circ \Delta = (\Delta \otimes \id) \circ \Delta$,
        \item The sets $(A\otimes 1)\Delta(A)$ and $(1\otimes A)\Delta(A)$ are linearly dense in $A\otimes_{\min}A$.
    \end{enumerate}
    \end{definition}

In the following, the algebra $A$ will be denoted by $C(\cqG)$ as it generalises the continuous functions on a compact group $G$.  In this classical case, $\Delta:C(G) \to C(G) \otimes_{\min} C(G)  = C(G \times G)$ is given by $\Delta f(s,t) = f(st)$ for $f \in C(G)$ and $s,t \in G$.  Then properties (i) and (ii) for $\Delta$ are just reformulations of the associativity and left/right cancellation properties (respectively), which completely characterise continuous group laws on compact Hausdorff spaces.  

Any compact quantum group admits an analogue of the unique Haar probability measure. See for example \cite[Theorem 1.2.1]{NT}.

\begin{theorem}
For any compact quantum group $\cqG$, there exists a unique state $h:C(\cqG)\to \C$, called the {\it Haar state}, satisfying 
    \begin{equation*}
        (\id\otimes h)(\Delta(a))=h(a)1=(h\otimes \id)(\Delta(a))
    \end{equation*}
    for all $a\in C(\cqG)$.
\end{theorem}

Two notions of representation will play a role in this work. The first is the usual notion of a (continuous) unitary representation derived from the theory of locally compact groups.  The second one corresponds to ordinary $*$-representations of the C*-algebra $C(\cqG)$. These two notions of representation are dual to each other in the sense of Pontryagin duality, as formulated in Proposition \ref{prop:qg-dual-representations}.  Below we will make use of the leg numbering notation, and denote by $B_0(\mathcal H)$ the C*-algebra of compact operators on a Hilbert space $\mathcal H$.   
\begin{definition}
    A \textit{unitary representation} of $\cqG$ on a Hilbert space $\mathcal{H}$ is a unitary element $u\in M(B_0(\mathcal{H})\otimes C(\cqG))$ such that 
    \begin{equation*}
        (\id\otimes \Delta)(u)=u_{12}u_{13},
    \end{equation*}
    where $M(A)$ denotes the C*-algebra of multipliers of a given C*-algebra $A$. 
    A closed subspace $\mathcal{K}\subset \mathcal{H}$ is called invariant if $(p\otimes 1)u(p\otimes 1)=u(p\otimes 1)$, where $p$ is the orthogonal projection onto $\mathcal{K}$. A representation $u$ is called \textit{irreducible} if the only invariant closed subspaces are $\{0\}$ and $\mathcal{H}$.
\end{definition}

The next definition is not standard, but it is a useful concept nontheless. Note that transpose maps with respect to different bases are related by unitary conjugation, so the definition is well-defined.
\begin{definition}
    Let $T: B(\mathcal{H})\rightarrow B(\mathcal{H})$ be the transpose map with respect to an orthonormal basis. A unitary representation $u$ of a compact quantum group $\cqG$ on a finite-dimensional Hilbert space $\mathcal{H}$ is called a \textit{bi-unitary representation} if $(T\otimes \id)(u)$ is also unitary.
\end{definition}

As with groups, there is a concept of intertwiners for compact quantum groups and a version of Schur's lemma (Lemma \ref{Schur}).
\begin{definition}
    Let $v,w$ be unitary representations of a compact quantum group $\cqG$ on $\mathcal{H}_1$ and $\mathcal{H}_2$, respectively. An \textit{intertwiner} between $v$ and $w$ is an element $S\in B(\mathcal{H}_1,\mathcal{H}_2)$ such that $(S\otimes 1)v=w(S\otimes 1)$. Two unitary representations are called \textit{unitarily equivalent} if there exists a intertwiner between them that is unitary.  The space of intertwiners from $u$ to $v$ is denoted $\text{Mor}(u,v)$.
\end{definition}
\begin{notation}
    $\mathrm{Irred}(\cqG)$ denotes the equivalence classes of irreducible (unitary) representations of $\cqG$. For any $x\in \mathrm{Irred}(\cqG)$ we fix a representative $u^x$ on Hilbert space $\mathcal{H}_x$. 
\end{notation}
\begin{lemma} \label{Schur}
    A unitary representation $u$ is irreducible if and only if all intertwiners between $u$ and itself are scalar multiples of the identity.
\end{lemma}

Denote by $\mathcal O(\cqG)$ the vector space spanned by the matrix coefficients of the irreducible representations of $\cqG$,  \textit{i.e.} 
\[\mathcal O (\cqG) = \text{span}\{u^x_{\xi,\eta} = (\omega_{\xi,\eta} \otimes \id)u^x: x \in \mathrm{Irred}(\cqG), \  \xi, \eta \in \mathcal H_x\} \subseteq C(\cqG),\]
where $\omega_{\xi,\eta}(a)=\braket{\xi}{a}{\eta}$ for $a \in B(\mc H_x)$.  $\mathcal O(\cqG)$ is a dense unital $*$-algebra of $C(\cqG)$, and the  comultiplication $\Delta$ restricts to a unital $*$-homomorphism $\Delta:\mathcal O(\cqG) \to \mathcal O(\cqG)  \otimes \mathcal O(\cqG) $, turning $\mathcal O(\cqG)$ into a  {\it Hopf $*$-algebra}.  The fact that $\mathcal O(\cqG)$ is a $*$-algebra follows from the existence of the tensor product, contragredient, and complete reducibility of finite dimensional representations.     It turns out that $\mathcal O (\cqG)$ is the unique dense Hopf $\ast$-subalgebra of $C(\cqG)$.  

Since $\mathcal O(\cqG)$ is spanned by coefficients of unitary operators, it admits a universal C*-completion, denoted by $C(\cqG_{\mathrm{max}})$.  The comultiplication $\Delta$ extends continuously to a comultiplication $\Delta_{\max}$ on $C(\cqG_{\max})$, making the pair $(C(\cqG_{\max}), \Delta_{\max})$ a compact quantum group, called the \textit{maximal version of $\cqG$}.  The Haar measure $h$ restricts to a faithful state on  $\mathcal O(G)$, and the resulting C*-completion of $\mathcal O(G)$ obtained by performing the GNS construction with respect to $h$ is denoted by $C(\cqG_{\min})$.  The comultiplication $\Delta$ on $\mathcal O(\cqG)$ extends continuously to a comulitplication $\Delta_{\min}$ on $C(\cqG_{\min})$, and the resulting compact quantum group is called the {\it minimal (or reduced) version of $\cqG$}.  

For any initial C*-realization $(C(\cqG), \Delta)$ of a compact quantum group $\cqG$, one has quotient maps 
\[
C(\cqG_{\max}) \to C(\cqG) \to C(\cqG_{\min})
\]
extending the  identity map on $\mathcal O(\cqG)$ and intertwining the comultiplications.  Since $\mathcal O(\cqG)$ is uniquely determined by {\it any} of these C*-completions and conversely any of these C*-completions can be recovered from $\mathcal O(\cqG)$, they are all just different C*-algebraic manifestations of a single compact quantum group structure $\cqG$.  This is analogous to the fact that a discrete group $\Gamma$ can be encoded C*-algebraically in more than one way (\textit{e.g.} via its full C*-algebra $C^*(\Gamma)$, or its reduced C*-algebra $C^*_r(\Gamma)$).

\subsubsection{Discrete quantum groups, duality, and property (T)} \label{subsubsec:dqg-duality-prop-T}

Discrete quantum groups can be defined axiomatically using the language of multiplier Hopf $*$-algebras \cite{vDa96}, or equivalently as structures dual to compact quantum groups.  We follow the latter approach, as it best suits our needs.

Let $\cqG$ be a compact quantum group.  Associated to $\cqG$, we define the C*-algebra 
\begin{equation*}
        C_0(\hat \cqG )=\bigoplus_{x\in \mathrm{Irred}(\cqG)}^{c_0}B(\mathcal{H}_x)
    \end{equation*}
and the von Neumann algebra
    \begin{equation*}
        \ell^{\infty}(\hat{\cqG})=\prod_{x\in \mathrm{Irred}(\cqG)}^{\ell^\infty} B(\mathcal{H}_x)
\end{equation*}
For $x \in 
\text{Irred}(\cqG)$, the minimal central projection associated to $B(\mathcal H_x) \subset C_0(\hat\cqG) \subset \ell^\infty(\hat \cqG)$ is denoted by $p_x$.
More generally, for $E \subseteq \text{Irred}(\cqG)$, we write $p_E = \sum_{x \in E} p_{x}$, with the sum converging $\sigma$-weakly in $\ell^\infty(\hat \cqG)$.  We also denote by $\Tr_x$ the canonical un-normalised trace on $B(\mathcal H_x)$, $\dim x = \Tr_x(1) = \dim \mathcal H_x$, and $\dim E = \sum_{x \in E} \dim x$ for any subset $E \subseteq \text{Irred}( \cqG)$.  At times, we will also work with $C_{00}(\hat \cqG) \subset C_{0}(\hat \cqG)$, the dense $*$-subalgebra of finitely supported elements.

There exists a normal injective co-associative $*$-homomorphism $\hat \Delta:\ell^\infty(\hat \cqG) \to \ell^\infty(\hat \cqG) \overline{\otimes} \ell^\infty(\hat \cqG)$ given by 
$\hat{\Delta}(ap_x)S=Sap_x$ for all $a\in \ell^\infty(\hat \cqG)$, $S \in \text{Mor}(u^x, u^y\otimes u^z)$, $x,y,z\in \mathrm{Irred}(\cqG)$. Equivalently, $\hat \Delta$ can be defined in terms of the unitary  
    \begin{equation*}
        \mathcal{V}=\bigoplus_{x\in \mathrm{Irred}(\cqG)}u^x\in M(C_0(\hat{\cqG})\otimes C(\cqG_{\max})).
    \end{equation*}
Then $\hat \Delta$ is uniquely determined by the identity 
\[(\hat \Delta \otimes \id) \mathcal{V}=\mathcal{V}_{13}\mathcal{V}_{23}.\]
Although we will have little need for them here, we mention that $\ell^\infty(\hat \cqG)$ comes equipped with left and right invariant weights, $\hat h_L$ and $\hat h_R$, satisfying the formal identities
\[(\hat h_R \otimes \id )\hat \Delta (a) = \hat h_R (a)1 \quad \&  \quad (\id \otimes \hat h_L )\hat \Delta (a) = \hat h_L(a)1 \qquad (a \in \ell^\infty(\hat \cqG)). \]

The quadruple $\dqG = \hat \cqG = (\ell^\infty(\hat \cqG), \hat \Delta, \hat h_L, \hat h_R)$ is the {\it discrete quantum group dual to $\cqG$}.  We also denote by $\ell^1(\hat \cqG) = (\ell^\infty(\hat \cqG))_*$, the predual of $\ell^\infty(\hat \cqG)$.  $\ell^1(\hat \cqG)$ is a completely contractive Banach algebra with convolution product given by 
\[
\psi \star \varphi = (\psi \otimes \varphi)\circ \hat \Delta \qquad (\psi, \varphi \in \ell^1(\hat \cqG)).
\]

A \textit{unitary representation} of $\hat \cqG$ on a Hilbert space $\mathcal{H}$ is a unitary $U\in M(C_0(\hat{\cqG})\otimes B_0(\mathcal H))$ such that 
    \begin{equation*}
        (\hat \Delta\otimes \id)(U)=U_{13}U_{23}.
        \end{equation*}

By faithfully representing $C(\cqG_{\max})$ on a Hilbert space $\mathcal H$, one can regard the unitary $\mathcal V$ defined above is a 
special example of a {\it unitary representation} of $\hat \cqG$.  $\mathcal V$ is a multiplicative unitary in the sense of \cite{BS93}.

For a discrete quantum group $\hat \cqG$, one can also define  co-unit and (generally unbounded) antipode maps. The {\it co-unit} is the normal state $\hat \epsilon:\ell^\infty(\hat \cqG) \to \C$ given by $\hat\epsilon(x) = xp_{0}$, where $p_0$ is the rank one central projection corresponding to the trivial representation of $\cqG$.  For our purposes, it suffices to  densely define the antipode $\hat S:\ell^\infty(\hat \cqG) \to \ell^\infty(\hat \cqG)$ by
$(\hat S \otimes \id)U = U^*$ for any unitary representation of $\hat \cqG$.

As one would expect, unitary representations of a discrete quantum group $\hat \cqG$ are in one-to-one correspondence with $*$-representations of the C*-algebra of the dual compact quantum group $\cqG$, and vice versa.  This duality is encoded precisely in terms of the multiplicative unitary $\mathcal V$.

\begin{proposition} \label{prop:qg-dual-representations}
    For any unitary representation $U$ of $\hat{\cqG}$ on $\mathcal{H}$ there exists a unique $*$-homomorphism $\pi: C(\cqG_{\max}) \to \mathcal{H}$ such that $(\id\otimes \pi)(\mathcal{V})=U$. Conversely, for any unitary representation $u$ of $\cqG$ on $\mathcal{H}$ there exists a unique $*$-homomorphism $\rho: \ell^{\infty}(\hat{\cqG})\rightarrow B(\mathcal{H})$ such that $(\rho\otimes \id)(\mathcal{V})=u$.
\end{proposition}
We conclude this recap of quantum groups by giving the definitions of property (T) and Kazhdan pairs for discrete quantum groups, which were introduced by Fima \cite{Fim10}.
    For a unitary representation $U$ of a discrete quantum group $\hat{\cqG}$ on $\mathcal{H}$, we write $U^x$ for $Up_x$ as element of $B(\mathcal{H}_x)\otimes B(\mathcal{H})$.  Below, in the quantum case, we deviate slightly from our previous notation $(S,\epsilon)$ for Kazhdan pairs, and instead write $(E,\epsilon)$. This is to avoid possible confusion with the antipodes $S$ and $\hat S$.

\begin{definition}
    Let $\hat{\cqG}$ be a discrete quantum group and $U$ a unitary representation of $\hat{\cqG}$ on $\mathcal{H}$. 
    \begin{itemize}
    \item We say that $U$ has an {\it invariant} vector if there exists a unit vector $\xi \in \mathcal H$ such that for all $x \in \text{Irred}(\cqG)$ and $\eta \in \mathcal H_x$, we have that
    \[U^x(\eta \otimes \xi) = \eta \otimes \xi.\]
        \item Let $E\subset\mathrm{Irred}(\cqG)$ be a finite subset and $\epsilon>0$. We say that $U$ has an \textit{$(E,\epsilon)$-invariant vector} if there exists a unit vector $\xi\in \mathcal{H}$ such that for all $x\in E$ and $\eta\in \mathcal{H}_x$ we have that
        \begin{equation*}
            \norm{U^x\eta \otimes \xi-\eta\otimes \xi}<\epsilon\norm{\eta}.
        \end{equation*}
        \item We say that $U$ has \textit{almost invariant vectors} if, for all finite $E\subset \mathrm{Irred}(\cqG)$ and all $\epsilon>0$, $U$ has an $(E,\epsilon)$-invariant vector.
        \item We say that $\hat{\cqG}$ has \textit{property (T)} if every unitary representation of $\hat{\cqG}$ having almost invariant vectors has a non-zero invariant vector.
        \item A pair $(E,\epsilon)$, where  $E \subset \mathrm{Irred}(\cqG)$ is finite and $\epsilon>0$ is called a 
  \textit{Kazhdan pair} if every unitary representation of $\hat{\cqG}$ having an $(E,\epsilon)$-invariant vector also has a non-zero invariant vector. 
    \end{itemize}
\end{definition}

As shown in \cite{Fim10}, discrete quantum groups with property (T) are {\it unimodular} and {\it finitely generated}.  Here unimodular means that the Haar measure $h$ on $\cqG$ is a tracial state.  Equivalently, unimodularity is characterised at the level of $\hat \cqG$ by the equality of the left and right Haar weights. In this case, the Haar weight $\hat h = \hat h_L = \hat h_R$ is the semifinite trace given by \[\hat h (ap_x) = \dim x \text{Tr}_{x}(ap_x) \qquad  (x \in \text{Irred}(\cqG), a \in \ell^\infty(\hat \cqG)).\] 
A discrete quantum group $\hat \cqG$ is unimodular if and only if $\mathcal O (\cqG)$ is finitely generated as a $*$-algebra.  Equivalently, there exists a finite subset $E\subset \mathrm{Irred}(\cqG)$ containing the trivial representation such that any finite dimensional unitary representation of $\cqG$ is generated by $E$ by taking direct sums, tensor products and subrepresentations.  Such a set $E$ called a generating set for $\hat \cqG$.   

As is in the case of ordinary discrete groups, discrete quantum groups with property (T) always admit a Kazhdan pair. 
    
\begin{proposition}[\cite{Fim10}]
    A discrete quantum group $\hat{\cqG}$ has property (T) if and only if there exists a finite generating set $E \subset \text{Irred}(\cqG)$ and  an $\epsilon>0$ such that $(E,\epsilon)$ is a Kazhdan pair.
\end{proposition}

A consequence of unimodularity for property (T) discrete quantum groups is the bi-unitarity of all finite-dimensional unitary representations of the compact dual quantum group.

\begin{proposition}
    Let $\cqG$ be a compact quantum group such that $\hat{\cqG}$ has property (T). Then all unitary representations of $\cqG$ are bi-unitary. 
\end{proposition}
\begin{proof}
Let $u = \sum_{ij} e_{ij} \otimes u_{ij} \in \mbb{M}_n(\C) \otimes \mathcal O(\cqG)$ be a finite dimensional unitary representation of $\cqG$, and let $T$ be the transpose map on $\mbb{M}_n(\C)$.   Put $v=(T\otimes \id)(u^*) =  \sum_{ij} e_{ij} \otimes u_{ij}^*$.  Then $v$ is a representation of $\cqG$ (the {\it contragredient of $u$}) and $v$ is clearly unitary iff $(T\otimes \id)(u)$ is unitary. 

On the other hand, by the proof of \cite[Proposition 6.4]{MvD98}, $R^{1/2}vR^{-1/2}$ is unitary, where $R = (\id\otimes h)(v^*v)$. But using the fact that the Haar measure $h$ is tracial, we obtain  
    \begin{equation*}
        R = (\id\otimes h)(v^*v)=\sum_{ijk}e_{ij}\otimes h(u_{ki}u_{kj}^*)=\sum_{ijk}e_{ij}\otimes h(u_{kj}^*u_{ki})=(T\otimes h)(u^*u)=1.\qedhere
    \end{equation*}
\end{proof}

\subsection{Quantum channels from quantum groups}  \label{subsec:q-channels-from-q-groups}

As in the case of groups discussed in Section \ref{background}, quantum groups and their representations naturally give rise to interesting quantum channels.  There are many ways in which one can construct quantum channels of various flavours from these more general algebraic structures -- see for example   \cite{Ve22, BCLY20, CN13, LY22}. Here we will just introduce the natural analogue of the mixed unitary channels associated to representations of finite or discrete groups.  We emphasise, however, that in this more general setup, the resulting quantum channels are not necessarily mixed unitary.  Below, we present two     constructions of quantum channels, from compact and discrete quantum groups, and show that due to the duality these constructions are equivalent. Ultimately, we see that the duality allows two perspectives on a class of channels, which can become useful when studying examples. 

We first begin with the construction of quantum channels from compact quantum groups.

\begin{proposition}\label{prop:cqg-quantum-channel}
    Let $\cqG$ be a compact quantum group, $u$ a finite-dimensional unitary representation of $\cqG$ on a Hilbert space $\mathcal{H}$, $\pi$ a $*$-representation of $C(\cqG_{\max})$ on a finite-dimensional Hilbert space $\mathcal{K}$, and $\phi$ a state on $B(\mathcal{H})$. Then $\Phi = \Phi_{u, \phi, \pi}: B(\mathcal{K})\rightarrow B(\mathcal{K})$, given by
    \begin{equation*}
        \Phi(x)=(\phi\otimes\id)\big((\id\otimes \pi)(u)(1\otimes x)(\id\otimes \pi)(u^*)\big),
    \end{equation*}
    is a normal unital completely positive map. If $u$ is a bi-unitary (\textit{e.g.} if $\cqG$ is of Kac type $\iff \hat \cqG$ is unimodular), then $\Phi$ is trace-preserving.
\end{proposition}
\begin{proof}
    The map $\Phi$ is a composition of completely positive maps, so it is completely positive. Next, 
    \begin{equation*}
        \Phi(1)=(\phi\otimes\id)\big((\id\otimes \pi)(u)(1\otimes 1)(\id\otimes \pi)(u^*)\big)=(\phi\otimes \id)\big((\id\otimes \pi)(uu^*)\big)=1,
    \end{equation*}
    so $\Phi$ is unital. If $u$ is bi-unitary, write $u$ as $u=\sum_{ij}e_{ij}\otimes u_{ij}$. Then
    \begin{equation*}
        \Tr(\Phi(x))=\sum_{ijkl}\phi(e_{ij}e_{lk})\Tr(\pi(u_{ij})x\pi(u_{kl}^*))=\sum_{ijk}\phi(e_{ik})\Tr(x\pi(u_{kj}u_{ij}^*))=\sum_i\phi(e_{ii})\Tr(x),
    \end{equation*}
   showing that $\Phi$ is trace-preserving, since $\phi$ is a state.
\end{proof}

One can equivalently describe the above class of ucp maps in terms of the dual discrete quantum group. Let $\dqG=\hat \cqG$ be a discrete quantum group, $U \in M(C_0(\dqG) \otimes B_0(\mathcal K))$ a unitary representation of $\dqG$ on a finite dimensional Hilbert space $\mathcal{K}$, and $\psi \in \ell^1(\dqG)$ a normal state. Then it is easy to see that $\Psi = \Psi_{\psi, U}: B(\mathcal{K})\rightarrow B(\mathcal{K})$, given by
    \begin{equation*}
        \Psi(x)=(\psi\otimes \id)(U(1\otimes x)U^*),
    \end{equation*}
is a normal unital completely positive map.  

In the next proposition, we will need the notion of a {\it finitely supported} element $\psi \in \ell^1(\hat \cqG)$: There exists a finite rank  central projection $p \in C_{00}(\hat \cqG)$ such that $\psi(xp) = \psi(x)$ for all $x \in \ell^\infty(\hat \cqG)$

\begin{proposition} \label{prop:dqg-quantum-channel}
    Let $\cqG$ be a compact quantum group. For every map $\Phi = \Phi_{u,\phi,\pi}$ arising in  Proposition \ref{prop:cqg-quantum-channel}, there exists a finitely supported state $\psi \in \ell^1(\hat \cqG)$ and a unitary representation $U$ of $\hat \cqG$ such that $\Phi = \Psi_{\psi, U}$.  Conversely, given a finitely supported state $\psi \in \ell^1(\hat \cqG)$ and a finite-dimensional unitary representation $U$ of $\hat \cqG$ on $\mathcal K$, then the UCP map $\Psi_{\psi,U}:B(\mathcal K) \to B(\mathcal K)$ is given by $\Psi_{\psi,U} = \Phi_{u,\phi, \pi}$  for $u, \phi, \pi$ as in Proposition \ref{prop:cqg-quantum-channel}.  

\end{proposition}
\begin{proof}
    Let $u$, $\pi$, $\phi$, and $\Phi$ be as in Proposition \ref{prop:cqg-quantum-channel}. By Proposition \ref{prop:qg-dual-representations}, $U=(\id\otimes \pi)(\mathcal{V})$ is a unitary representation of $\hat{\cqG}$ and there exists a (normal) *-homomorphism $\rho: \ell^{\infty}(\hat \cqG)\rightarrow B(\mathcal{H})$ such that $(\rho\otimes\id)(\mathcal{V})=u$. Now $\psi=\phi\circ \rho$ is a normal state in $\ell^{1}(\hat{\cqG})$ and we find
    \begin{equation*}
        \Phi(x)=(\phi\otimes \id)\big((\rho\otimes \pi)(\mathcal{V})(1\otimes x)(\rho\otimes \pi)(\mathcal{V}^*)\big)=(\psi\otimes \id)(U(1\otimes x)U^*) = \Psi(x).\qedhere
    \end{equation*}  Note that $\psi$ finitely supported because it is supported on the central summands of $\ell^\infty(\hat \cqG)$ associated to the irreducible subrepresentations of $u$.

Conversely, if we start with a pair $\{\psi, U\}$ as in the statement of the proposition, we obtain a unique morphism $\pi:C(\cqG_{\max}) \to B(\mathcal K)$ from $U$ via Pontryagin duality, and we obtain a pair $\{\phi,u\}$ from the state $\psi$ via the GNS representation.  Note that the finite support condition on $\psi$ endures that $u$ is a finite dimensional representation of $\cqG$.
\end{proof}

We  end this section with a description of the fixed points of the ucp maps $\Psi_{\psi, U}$ for the cases that will concern us.  To do this, we recall a well-known result about fixed points of quantum channels.

\begin{theorem}{\cite[Theorem 4.25]{Wat15}}\label{thm:commutant}
    Let $\Phi:\mbb{M}_N\rightarrow\mbb{M}_N$ be a unital CPTP map with Kraus decomposition $\Phi(\rho)=\sum_i K_i\rho K_i^\ast$. The set of fixed points of $\Phi$ is the commutant of the Kraus operators $\{K_i\}'$.
\end{theorem}
Note that this result is independent of the choice of Kraus representation.

Now let $\hat \cqG$ be a discrete quantum group with {\it finite} generating set $E \subset \text{Irred}(G)$, and let $U \in M(C_0(
    \hat \cqG
    ) \otimes B_0(\mathcal K))$ a finite-dimensional representation.   Let $p_E = \sum_{x \in E} p_x \in C_{00}(\hat \cqG)$ be the largest central projection supported on  $E$, and let $\psi \in \ell^1(\hat \cqG)$ be a state supported on $E$. 
 That is, $\psi(p_Ex) = \psi(x)$ for all $x \in \ell^\infty(\hat \cqG)$.

 In the following two results, we consider the associated normal UCP map $\Psi_{\psi,U}:B(\mathcal K) \to B(\mathcal K)$.  This map should be interpreted as the  quantum group generalization of the mixed  unitary channel 
 \begin{align}\label{eqn:random-unitary}\rho \mapsto \sum_{s \in E} \psi(s)\pi(s)\rho\pi(s)^*,\end{align}
 where $\psi$ is some probability density supported on a generating set $E$ of a group $\Gamma$. 

\begin{theorem}\label{thm:fixed-points}
Assume that the restriction 
 of $\psi$ to $p_E\ell^\infty(\hat \cqG)p_E = \bigoplus_{s \in E} B(\mathcal H_s)$ is faithful and that $\Psi_{\psi,U}$ is trace-preserving (which holds automatically if $\hat \cqG$ is unimodular).  Then  $\rho\in B(\mathcal K)$ is a fixed point of $\Psi_{\psi,U}$ if and only if it is an intertwiner of $U$.
\end{theorem}

\begin{proof}

Let $U^s \in B(\mathcal H_s) \otimes B(\mathcal K)$ denote the $s$-th component of $U$ for $s \in \text{Irred}(\cqG)$. Also let $\psi_s$ be the restriction of $\psi$ to $B(\mathcal H_s)$.  Choosing appropriate bases of matrix units $e_{ij}^s$ of $B(\mathcal H_s)$, we may assume that the density of $\psi_s$ is diagonal and so $\psi_s(e_{ij}) = \lambda_{i,s}\delta_{i,j} > 0$.  (These coefficients are all non-zero by the faithfulness assumption.)  Then we compute 
\begin{align*}
        \Psi_{\psi,U}(\rho)&=\sum_{s\in E}\sum_{a,b,c} \psi_s(e_{ab}^s)U^s_{ac}\rho(U^s_{bc})^\ast\\
        &=\sum_{s\in E}\sum_{a,c} \lambda_{a,s}U^s_{ac}\rho(U^s_{ac})^\ast .
    \end{align*}
    By \cref{thm:commutant}, $\rho$ is a fixed point of $\Psi_{\psi,U}$ if and only if $\sqrt{\lambda_{a,s}}U^s_{ac}\rho=\rho\sqrt{\lambda_{a,s}}U^s_{ac}$ for all $s\in E$ and $1\leq a,c \le \dim \mathcal H_s$. As $E$ is generating, we conclude that $(1 \otimes \rho)U = U(1 \otimes \rho)$.  The reverse implication is immediate.
\end{proof}

An immediate consequence of interest in the next section is:

\begin{corollary} \label{cor:mulitplicity}
    If $U$ is an irreducible finite dimensional  representation of a unimodular discrete quantum group $\hat \cqG$, the eigenvalue $1$ of the unital channel $\Psi_{\psi,U}$ has multiplicity~$1$.
\end{corollary}

\section{Quantum expanders from property (T) quantum groups } \label{sec:edge-expanders}

Throughout this section we fix a discrete quantum group $ \hat \cqG$ with property (T), along with a  fixed Kazhdan pair $(E,\epsilon)$ for $\hat \cqG$.  We furthermore can and will always take $E$ to be {\it symmetric}. That is,  $E = \bar E$, where $\bar E = \{\bar x: x \in E\}$ and $\bar x$ denotes the conjugate representation associated to $x \in \text{Irred}(\cqG)$.  In the following proposition, we let $\mathcal H_E = \bigoplus_{x \in E} \mathcal H_x$ and consider the natural embedding $p_E\ell^\infty(\hat \cqG) p_E\subset B(\mathcal H_E)$.  Let $\psi_E$ be the normal tracial state on $\ell^\infty(\hat \cqG)$ defined by $\psi_E(a)=\frac{1}{\dim \mathcal H_E}\Tr_{B(\mathcal H_E)}\parens*{ap_E}$. Note that $\psi_E$ is nothing other than the normalised trace on $\mathcal H_E$, lifted to $\ell^\infty(\hat \cqG)$ in the obvious way.  We are interested in studying the expansion properties of the channel \begin{align}\label{eqn:Kazhdan-channel}\Psi_{\psi_E,U}:B(\mathcal K) \to B(\mathcal K), \qquad \Psi_{\psi_E,U}(\rho) = (\psi_E \otimes \id)(U(1 \otimes \rho)U^*) \qquad (\rho \in B(\mathcal K)),
\end{align}
where $U\in M(C_0(\hat \cqG) \otimes B_0(\mathcal K))$ is a finite-dimensional irreducible unitary representation of $\hat \cqG$ on $\mathcal K$.  

We begin with a definition and a couple of observations. 

\begin{definition}\label{def:symmetric-state}
A normal linear functional $\psi \in \ell^1(\hat \cqG)$ is called {\it symmetric} if
\[
\psi\circ \hat S = \psi.
\]
\end{definition}

\begin{lemma}\label{lem:self-adjoint}
Let $\psi \in \ell^1(\hat \cqG)$ be a symmetric state.
For any unitary representation $V\in M(C_0(\hat \cqG) \otimes B_0(\mathcal K))$, the map $X = (\psi \otimes \id)V \in B(\mathcal K)$ is self-adjoint.
\end{lemma}

\begin{proof}
Note that  $\hat S$ is (densely) defined by the equations $(\hat S \otimes \id)V  = V^*$ for all unitary representations of $\hat \cqG$.  Because $\psi$ is positive, we have
\[
X^* = ((\psi \otimes \id)V)^* = (\psi \otimes \id)(V^*) = (\psi\circ \hat S \otimes \id) V = (\psi \otimes \id) V = X.
\qedhere\]
\end{proof}

Note that the above lemma applies to the case where $\psi = \psi_E$,  since $\psi_E \circ \hat S = \psi_{\bar E} = \psi_E$.  Here we use the symmetry condition $E =\bar E$

\begin{lemma} \label{lem:technical}
    Let $\mathcal H$ be a finite dimensional Hilbert space and $0 < \epsilon < 1$.   Suppose $A\in B(\mathcal H)$ satisfies $\norm{A}\leq 1$ and there exists a unit vector $\eta \in \mathcal H$ such that $\latRe \braket{\eta}{A\eta}\leq 1-\epsilon$. Then, \[\latRe\tr(A)\leq1-\frac{\epsilon}{\dim \mathcal H}.\]
\end{lemma}

\begin{proof}
    Write $d=\dim \mathcal H$. Extend $\eta$ to an orthonormal basis $\eta_1 = \eta,\eta_2,\ldots,\eta_d$ for $\mathcal H$. Then,
    \[\latRe\tr(A)=\frac{1}{d}\latRe(\parens*{\braket{\eta_1}{A}{\eta_1}+\ldots+\braket{\eta_d}{A}{\eta_d}})\leq\frac{1}{d}\parens*{1-\epsilon+d-1}=1-\frac{\epsilon}{d}.\qedhere\]
\end{proof}

We now come to an estimate which will lead to the main result of the section.  

\begin{proposition}\label{prop:main-estimate} Let $V\in M(C_0(\hat{\cqG})\otimes B_0(\mathcal K))$ be an unitary representation 
    of $\hat \cqG$ on a Hilbert space $\mathcal K$.  Let $\mathcal K_0 \subset \mathcal K$ be the closed invariant subspace of invariant vectors for $V$. Then, for any unit vector $\xi \in \mathcal K_0^\perp$, \[\latRe\braket{\xi}{(\psi_E \otimes \id )(V)}{\xi}\leq 1-\frac{\epsilon^2}{2\dim \mathcal H_E}.\]
\end{proposition}

\begin{proof}
   Consider the restriction of $V$ to   $\mathcal K_0^\perp$, on which it remains a representation of $\hat \cqG$, but now without invariant vectors. Fix a unit vector $\xi\in \mathcal K_0^\perp$.  Since $(E, \epsilon)$ is a Kazhdan pair for $\hat \cqG$, there exists some unit vector $\eta\in \mathcal H_E$ such that $\norm{V\eta \otimes \xi-\eta \otimes \xi}\geq\epsilon$. This gives
    $$\epsilon^2\leq\norm{V\eta \otimes \xi-\eta \otimes \xi}^2=2-2\latRe\bra{\eta \otimes \xi}V\ket{\eta \otimes \xi} \iff \latRe\bra{\eta \otimes \xi}V\ket{\eta \otimes \xi} \le 1-\frac{\epsilon^2}{2} $$
    Applying Lemma \ref{lem:technical} to the contraction $A = (\id \otimes \omega_{\xi,\xi})V \in B(\mathcal H_E)$ and noting that $\psi_E$ is the normalised trace on $\mathcal H_E$, we obtain the result.
\end{proof}

Applying a Frobenius reciprocity argument, we can use Proposition \ref{prop:main-estimate} to immediately deduce the spectral gap of the channels of interest $\Psi_{\psi_E,U}$. 

\begin{theorem} \label{thm:lambda2}
    Let $U \in M(C_0(\hat \cqG) \otimes B_0(\mathcal K))$ be a finite-dimensional irreducible representation. 
    Then 
    \[\lambda_2(\Psi_{\psi_E,U}) \le 1-\frac{\epsilon^2}{2\dim \mathcal H_E}.\]
\end{theorem}

\begin{proof}
Consider the unitary representation $V = U \otimes \bar U$ on the Hilbert space $\mathcal K \otimes \bar {\mathcal K}$.  Here $\bar U$ is the conjugate representation of $U$ acting on $\bar{ \mc K}$ associated to the discrete quantum group (see for example \cite[Definition 2.2]{DSV17} for the construction of $\bar{U}$).  Under the canonical unitary isomorphism of Hilbert spaces $B(\mathcal K) \cong \mathcal K \otimes \bar {\mathcal K}$; $\ketbra{\xi}{\eta} \mapsto \xi \otimes \bar\eta$ (where $B(\mathcal K)$ is equipped with the usual trace-inner product), we have that $\Psi_{\psi_E, U}$ is identified with the self-adjoint map 
\[Y = (\psi_E \otimes \id)(V)  \in B(\mathcal K \otimes \bar{\mathcal K}). \] 
Now, by self-adjointness of  $\Psi_{\psi_E,U}$ (or, equivalently, $Y$), we have
\[\lambda_2(\Psi_{\psi_E,U}) = \sup\Big\{\latRe\Tr\parens*{X^\ast\Psi_{\psi_E,U}(X)}\Big\},\]
where the supremum is taken over  $X\in B (\mc{K})$ such that $\norm{X}_2=1$ and $X$ is orthogonal to the $1$ eigenspace of $\Psi_{\psi_E,U}$.  By Corollary \ref{cor:mulitplicity},   this eigenspace is exactly $\mathbb C 1$, and corresponds to the one-dimensional subspace $(\mathcal K \otimes \bar{\mathcal K})_0$ of invariant vectors for $V$. 
Thus, we may apply Proposition \ref{prop:main-estimate} to deduce \[\latRe\Tr\parens*{X^\ast\Psi_{\psi,U}(X)}\leq1-\frac{\epsilon^2}{2\dim \mathcal H_E},\] for all such $X$.
\end{proof}

Applying Theorem  \ref{thm:quantum-cheeger}, we immediately obtain the following

\begin{corollary} \label{cor:lamda2-to-h}
The channels $\Psi_{\psi_E,U}$ considered in Theorem \ref{thm:lambda2} have quantum edge expansion \[h_Q(\Psi_{\psi_E,U})\geq\frac{\epsilon^2}{4\dim \mathcal H_E}\] in the sense of~\cite{LQWWZ22arxiv}.
\end{corollary}

Note that the lower bound on the expansion constant is independent of the size of the Hilbert space $\mathcal K$ on which the representation $U$ acts. Moreover, the channels $\Psi_{\psi_E,U}$ have degree uniformly bounded by $|E|:=\sum_{x \in E} (\dim \mathcal H_x)^2$.  Indeed, this follows from the fact that, after choosing appropriate bases, $\Psi_{\psi_E,U}$ has the explicit  Kraus form 
\[\Psi_{\psi_E,U}(\rho) = \frac{1}{\dim \mathcal H_E}\sum_{s \in E}\sum_{1 \le i,j \le \dim x} U_{ij}^x \rho (U_{ij}^x)^* \qquad (\rho  \in B(\mathcal K)), \]
where $U^x = (U^x_{ij}) \in B(\mathcal H_x) \otimes B(\mathcal K)$ is the block of $U$ associated to a given $x \in \text{Irred}(\cqG)$, as in the proof of~\cref{thm:fixed-points}. 

In particular, the problem of construction of explicit bounded degree expander families reduces to that of finding an increasing family of finite-dimensional irreducible representations of $\hat \cqG$.

\begin{theorem}\label{thm:quantum-margulis}
Let $\hat \cqG$ be an infinite property (T) discrete quantum group with symmetric Kazhdan pair $(E,\epsilon)$.  Let $(U_i)_{i\in \mathbb N}$ be a sequence of finite-dimensional irreducible unitary representations on Hilbert spaces $\mathcal K_i$ with $\lim_i \dim \mc K_i = \infty$.  Then the family of bistochastic quantum channels $(\Psi_i)_{i\in \mathbb N}$, with $\Psi_i = \Psi_{\psi_E, U_i}$ forms a bounded degree quantum expander family with quantum edge expansion \[\inf_{i \in \mathbb N} h_Q(\Psi_i) \geq\frac{\epsilon^2}{4\dim \mathcal H_E}.\]      
\end{theorem}

The conditions of Theorem \ref{thm:quantum-margulis} are of course immediately satisfied if $\hat \cqG$ is {\it residually finite} in the sense of \cite{BBCY20}, that is, when $\mathcal O(\cqG)$  is residually finite dimensional as a $*$-algebra.     

\begin{corollary}
Let $\hat \cqG$ be an infinite discrete quantum group with property (T).  If $\hat \cqG$ is residually finite, then $\hat \cqG$ gives rise to a bounded degree family of quantum  expanders.
\end{corollary}

\begin{remark} \label{rem:factorization}
Note that in the presence of property (T), the results of \cite{BBCY20} show that residual finiteness of $\hat G$ is in fact equivalent to $\hat \cqG$ having the a priori weaker {\it Kirchberg Factorization Property}.  In terms of the Haar state on $\cqG$, this latter condition is equivalent to the Haar state $h:C(\cqG_{\max}) \to \mathbb C$  being a so-called {\it amenable trace}.  At this time, it is not clear to us whether checking for the Factorization Property is any easier than directly verifying residual finiteness.  In the next section, we will explore examine some natural examples quantum groups with property (T), arising from compact bicrossed products of residually finite discrete groups, and their associated expanders.
\end{remark}

\subsection{Spectral gap for channels associated to non-tracial states} \label{sec:non-tracial-gap}  In the previous section, we only considered expansion properties of the channels $\Psi_{\psi_E, U}:B(\mathcal K) \to B(\mathcal K)$, where $U$ is an irreducible finite-dimensional unitary representation of $\hat \cqG$ on $\mathcal K$, and $\psi_E$ is the tracial state on $\ell^\infty(\hat \cqG)$ associated to the normalised trace on $B(\mathcal H_E)$.  In this section, we outline how much of the same analysis may be extended if one replaces $\psi_E$ with any faithful symmetric state $\psi$ supported on $p_E\ell^\infty(\hat \cqG)p_E \subset B(\mathcal H_E)$.  We shall see from our estimates below that any deviation of $\psi$ from the canonical trace $\psi_E$ gives rise to weaker lower bound on the spectral gap.  This might suggest that the tracial case is optimal when considering expansion properties.  On the other hand, the weaker bound obtained below in the non-tracial case may just be a consequence of our crude estimates and could possibly be improved.  

The main technical ingredient we need is a non-tracial extension of Lemma \ref{lem:technical}.

\begin{lemma} \label{lem:technical2}
    Let $\mathcal H$ be a finite dimensional Hilbert space, $0 < \epsilon < 1$, and $\rho$ a density matrix in $B(\mathcal H)$ with smallest eigenvalue $\lambda$. Suppose $A\in B(\mathcal H)$ satisfies $\norm{A}\leq 1$ and there exists a unit vector $\eta \in \mathcal H$ such that $\latRe \braket{\eta}{A\eta}\leq 1-\epsilon$. Then, \[\latRe\Tr(\rho A)\leq1-\lambda \epsilon.\]
\end{lemma}

\begin{proof}
    Let $d= \dim \mathcal H$. Extend $\eta$ to an orthonormal basis $\eta_1=\eta,\dots,\eta_d$ for $\mathcal H$. Write $\rho=\lambda 1 + (1-d \lambda) \rho'$, where $\rho'=\frac{\rho-\lambda 1}{1-d\lambda}$ is also a density matrix. Then,
    \begin{equation*}
        \latRe \Tr(\rho A)=\lambda \latRe\Tr(A)+(1-d\lambda)\latRe\Tr(\rho' A)\leq 1-d\lambda+\lambda \latRe\sum_{i=1}^d\bra{\eta_i}A\ket{\eta_i}\leq 1-\epsilon \lambda. \qedhere
    \end{equation*}
    
\end{proof}

Let  $\psi:B(\mathcal H_E) \to \C$ be a fixed faithful state, which we also regard as a normal state on $\ell^\infty(\hat{\cqG})$ with central support $p_E$. Assume that $\psi$ is symmetric.  Write $\psi = \Tr(\sigma \cdot)$ for a  unique density $\sigma \in B(\mathcal H_E)$, and let $0 < \lambda \le \frac{1}{\dim \mathcal H_E}$ be the minimal eigenvalue of $\sigma$.

\begin{theorem} \label{thm:lambda2-nontracial}
Let $U \in M(C_0(\hat \cqG) \otimes B_0(\mathcal K))$ be a finite-dimensional irreducible representation. 
Then  for the channel 
\begin{align*}
\Psi_{\psi,U}:B(\mathcal K) \to B(\mathcal K), \qquad \Psi_{\psi,U}(\rho) = (\psi\otimes \id)(U(1 \otimes \rho)U^*) \qquad (\rho \in B(\mathcal K)),
\end{align*}\[\lambda_2(\Psi_{\psi,U}) \le 1-\frac{\lambda\epsilon^2}{2}.\]
\end{theorem}

\begin{proof}
As in the case of $\psi = \psi_E$ studied in the previous subsection, it suffices to show  that \begin{align}\label{eqn:non-tracial-basic}\latRe\braket{\xi}{(\psi \otimes \id )(V)}{\xi}\leq 1-\frac{\lambda\epsilon^2}{2},\end{align}  for any representation $V \in M(C_0(\hat \cqG) \otimes B_0(\mathcal K))$ and any unit vector $\xi \in \mathcal (\mathcal K_0)^\perp$, the complement of the fixed vectors.  This latter estimate follows immediately from Lemma \ref{lem:technical2} and property (T).
\end{proof}

\begin{remark}
One may also observe that the channels $\Psi_{\psi, U}$ still have degree bound given by $|E| =\sum_{x\in E} (\dim \mathcal H_x)^2$.  
\end{remark}

\section{Expanders associated to bicrossed product quantum groups}
\label{sec:bicrossed}

In this section, we apply the general theory of the previous section to analyse concrete examples of quantum expanders.  As we have seen in Theorem \ref{thm:quantum-margulis}, the construction of an explicit expander family simply requires as input a residually finite discrete quantum group $\hat \cqG$ with property~(T).  While there are nowadays many new and interesting examples of genuine infinite discrete quantum groups with property (T) \cite{VaVa19, FMP17, RV24}, surprisingly little seems to be known about their representation theory and their residual finiteness.  One class of examples where these properties are nevertheless well-understood are quantum groups arising from bicrossed products.  
We now very briefly introduce the notion of a (compact) bicrossed product quantum group. For more details we refer to \cite[Section 3]{FMP17}.

Let $G$ and $\Gamma$ be a compact and a discrete group, respectively,
forming a {\it matched pair} in the sense that they are realised as
trivially-intersecting closed subgroups of a locally compact group
$H$, with the property that the product $\Gamma G$ has full Haar
measure in $H$. This amounts to giving a left action $\alpha$ of
$\Gamma$ on $G$ and a right action $\beta$ of $G$ on $\Gamma$
satisfying certain compatibility conditions (\textit{i.e.} \cite[Proposition
3.3]{FMP17}).

From a quadruple $(\Gamma,G,\alpha,\beta)$ as above, one can construct a
compact quantum group $\cqG=\cqG(\Gamma,G,\alpha,\beta)$, as explained in \cite{VV03} or
\cite[Section 3.2]{FMP17}; we denote it by $\cqG(\Gamma,G,\alpha,\beta)$
when we wish to be explicit about the matched pair structure, and
reserve the present notation of $\Gamma$, $G$, $\alpha$, $\beta$, and
$\cqG$ throughout the current section. The dual discrete quantum groups $\hat \cqG$ will
provide, under certain conditions, non-trivial examples of discrete quantum groups with property (T).

Let us now discuss the construction of $\cqG$.  Let $A_m=\Gamma\ltimes_{\alpha,f} C(G)$ be the full crossed product C*-algebra and $A=\Gamma\ltimes_{\alpha} C(G)$ the reduced crossed product C*-algebra. Let $\omega$ denote the canonical injective maps from $C(G)$ to $A_m$ and from $C(G)$ to $A$. For $\gamma \in \Gamma$ we denote by $u_{\gamma}$ the canonical unitaries in either $A_m$ or $A$.  The bicrossed product construction endows $A_m$ with a comultiplication making it a compact quantum group.  This is done as follows.  First, for each coset $\gamma \cdot G\in \Gamma/G$, we define  
\begin{equation*}
    v^{\gamma\cdot G}=(v_{rs})_{r,s\in \gamma\cdot G}\in \mbb{M}_{|\gamma\cdot G|}\otimes C(G),
\end{equation*}
where $v_{rs}=\chi_{A_{r,s}}$ is the characteristic function of the set $A_{r,s}=\{g\in G|\beta_g(r) = s\}$. Note that  $v^{\gamma\cdot G}$ is a {\it magic unitary} over $C(G)$, that is, a matrix whose rows and columns are projection-valued measures (PVMs). Finally, the comultiplication is given by the unique unital *-homomorphism $\Delta_m : A_m\rightarrow A_m\otimes A_m$ such that
\begin{equation}
    \Delta_m\circ \omega=(\omega\otimes \omega)\circ \Delta_G\text{ and } \Delta_m(u_{\gamma})=\sum_{v\in\gamma\cdot G}u_{\gamma}\omega(v_{\gamma,r})\otimes u_r \qquad (\gamma\in \Gamma).
\end{equation}
Then $\cqG=(A_m,\Delta_m)$ is a compact quantum group, called the {\it (compact) bicrossed product} associated to the matched pair $(\Gamma, G)$.

The final piece of information we need 
 concerns the description of the  irreducible unitary representations of $\cqG$.  For each $\gamma \in \Gamma$, we consider the orbit $\gamma \cdot G \in \Gamma/G$  and define \[V^{\gamma\cdot G}=\sum_{r,s\in \gamma\cdot G}e_{r,s}\otimes u_r\omega(v_{r,s})\in \mbb{M}_{|\gamma\cdot G|}\otimes A_m.\] Then $V^{\gamma\cdot G}$ determines an irreducible unitary representation of $\cqG$, and the family $\{V^{\gamma\cdot G}\}_{\gamma\cdot G \in \Gamma/G}$ of representations are pairwise inequivalent and satisfy $\overline{V^{\gamma\cdot G}}\simeq V^{\gamma^{-1}\cdot G}$.  Moreover, any irreducible unitary representation of $\cqG$ is equivalent to a subrepresentation of $V^{\gamma\cdot G}\otimes v^x$ for some $\gamma\cdot G\in \Gamma/G$ and $x\in \mathrm{Irred}(G)$, where $v^{x}=(\id\otimes \omega)(u^x)$. 
 
\subsection{Property (T) and the bicrossed product construction}
To the best of our knowledge, a complete classification of property (T) for the duals $\hat \cqG$ of a compact bicrossed product $\cqG(\Gamma,G, \alpha,\beta)$ in terms of the input data is not known.  However, several important results have been obtained \cite{FMP17, VaVa19}.  In particular, \cite[Theorem 4.3]{FMP17} shows that $\hat \cqG$ has property (T) whenever $\Gamma$ has property (T) and $G$ is finite. Moreover, in this case, the proof of \cite[Theorem 4.3]{FMP17} shows that if $(E,\epsilon)$ is a Kazhdan pair for $\Gamma$ containing the identity element, and $v$ is a fundamental representation of $G$, then (the irreducible components of) $\{V^{\gamma \cdot G} \otimes v\}_{\gamma \in E}$ is a Kazhdan set for $\hat \cqG$. 

As a concrete example, one can take any natural number $n \ge 3$ and any prime number $p\ge 3$. Let $\Gamma = SL_n(\mathbb Z)$, $G = SL_n(\mathbb F_p)$, and consider the action $\alpha$ given by $\alpha_\gamma (g) = [\gamma]g[\gamma]^{-1}$, and $\beta$ being the trivial action.
The resulting discrete quantum group $\hat \cqG$ then has property (T).  Moreover, $\hat \cqG$ is residually finite by \cite[Theorem 4.2]{BBCY20}.

In the work \cite{VaVa19}, a very large class of compact bicrossed products $\cqG$ whose duals have property (T) was discovered using triangle presentations and related constructions.  These examples are somewhat more interesting than the above one, because they arise from matched pairs where $G$ is also infinite (and therefore $\hat \cqG$ is not commensurable with a classical property (T) group).  As an example (cf. \cite[Remark 6.3]{VaVa19}), one can let $\mathbb K$ be a commutative local field with ring of integers $\mathcal O$. Put $H = PGL(3, \mathbb K)$ and
$G = PGL(3, \mathcal O)$. Let $\Gamma < H$ be a subgroup such that $H = \Gamma G$ and $\Gamma \cap G = \{e\}$. Then the dual 
 of the compact bicrossed product $\cqG$ associated to the matched pair $(\Gamma,G)$ has property (T).  

\subsection{Channels corresponding to the bicrossed product construction}
Let us briefly describe the general flavour of the channels that one can expect from bicrossed product quantum groups.  We will use the notation of the previous section. Let $U$ be a finite dimensional unitary representation  of $\hat \cqG$ on a finite-dimensional Hilbert space $\mathcal K$, and let  $\pi:C(\cqG) \to B(\mathcal K)$ be the associated $\ast$-homomorphism. 

Let $\gamma\in \Gamma$ be arbitrary, and let us analyze what kind of channels we can obtain if we use the construction in Proposition \ref{prop:cqg-quantum-channel} with $E=\{V^{\gamma\cdot G}\}$. Let $\phi$ be a state on $\mbb{M}_{|\gamma\cdot G|}$. Using Proposition \ref{prop:cqg-quantum-channel}, we find the quantum channel $\Phi = \Phi_{ V^{\gamma \cdot G}, \phi, \pi}$ is given by
\begin{align*}
    \Phi(\rho)&=(\phi\otimes \id)((\id\otimes \pi)(V^{\gamma\cdot G})(I\otimes \rho)(\id\otimes \pi)(V^{\gamma\cdot G})^*)\\
    &=(\phi\otimes \id)\left(\sum_{r_1,r_2,s_1,s_2\in\gamma\cdot G}(e_{r_1,s_1}\otimes \pi(u_{r_1}\omega(v_{r_1,s_1})))(I\otimes \rho)(e_{r_2,s_2}\otimes \pi(u_{r_2}\omega(v_{r_2,s_2})))^*\right)\\
    &=(\phi\otimes \id)\left(\sum_{r_1,r_2,s\in\gamma\cdot G}e_{r_1,r_2}\otimes(\pi(u_{r_1}\omega(v_{r_1,s}))\rho \pi(u_{r_2}\omega(v_{r_2,s}))^*)\right)\\
    &=(\phi\otimes \id)\left(\sum_{r_1,r_2,s\in\gamma\cdot G}e_{r_1,r_2}\otimes(\pi(u_{r_1})\pi(\omega(v_{r_1,s}))\rho \pi(\omega(v_{r_2,s}))^*\pi(u_{r_2}))^*)\right).
\end{align*}
\begin{proposition}\label{prop:mixed-unitary}
If $\phi$ is a trace, then $\Phi$ is a mixed unitary channel.  
\end{proposition}
To prove this proposition, we first need an easy lemma. 
\begin{lemma} \label{lem:easy}
    Let $p_1,\dots p_n$ be orthogonal projections in $\mbb{M}_N$ that sum to the identity. Then the channel
    \begin{equation*}
        \rho\mapsto \sum_{i=1}^np_i\rho p_i
    \end{equation*}
    is a mixed unitary channel.
\end{lemma}
\begin{proof}
    Define the operator
    \begin{equation*}
        U=\sum_{k=1}^ne^{2\pi i k/n}p_k.
    \end{equation*}
    This is clearly a unitary. Moreover, 
    \begin{align*}
        \frac{1}{n}\sum_{k=1}^nU^k\rho (U^k)^*&=\frac{1}{n}\sum_{k,l,m=1}^ne^{2\pi i kl/n}p_l\rho e^{-2\pi i km/n } p_m\\
        &=\frac{1}{n}\sum_{k,l,m=1}^ne^{2\pi i k(l-m)/n}p_l\rho  p_m\\
        &=\sum_{l=1}^np_l\rho p_l.
    \end{align*}
    This shows that the channel is indeed a mixed unitary channel.
\end{proof}

\begin{proof}[Proof of Proposition \ref{prop:mixed-unitary}]
We will now assume that $\phi$ is the trace. In this case, the equation for $\Phi$ above Proposition \ref{prop:mixed-unitary} becomes
\begin{equation} \label{eq:tracial-bicrossed-product-channel}
    \Phi(\rho)=\frac{1}{|\gamma \cdot G|}\sum_{r,s\in \gamma\cdot G}\pi(u_r)\pi(\omega(v_{r,s}))\rho \pi(\omega(v_{r,s}))^*\pi(u_r)^*.
\end{equation}
By Lemma \ref{lem:easy} the channels
\begin{equation*}
    \rho \mapsto \sum_{s\in \gamma\cdot G}\pi(\omega(v_{r,s}))\rho \pi(\omega(v_{r,s}))^* \qquad (s \in \gamma \cdot G)
\end{equation*}
are mixed unitaries, which in turn shows that $\Phi$ is a convex combination of  compositions of unitary channels, and hence itself mixed unitary.     
\end{proof}

\begin{remark}
The conclusion of Proposition \ref{prop:mixed-unitary} holds more generally if we consider channels built from any direct sums of building block representations of the form $V^{\gamma \cdot G} \otimes v^x$, ($\gamma \in \Gamma$, $x \in \text{Irred}(G)$) and traces on the underlying multi-matrix algebras.  Indeed, the channels associated to single tensor products $V^{\gamma \cdot G} \otimes v^x$ and traces are easily seen to be compositions of mixed unitary channels, and passing to direct sums amounts to convex combinations of mixed unitary channels.
\end{remark}

\begin{remark}
Proposition \ref{prop:mixed-unitary} does not rule out the possibility of obtaining non-mixed unitary expanders from bicrossed product quantum groups when we work with non-tracial states $\phi$ on the spaces $\mbb{M}_N$.  However, we know from the results of  Section \ref{sec:non-tracial-gap} that our bounds for spectral gap are governed by the smallest eigenvalue of a $\phi$ (i.e., the tracal part of $\phi$ in the proof of Lemma \ref{lem:technical2}), so even if the resulting channel $\Phi$ was non-random unitary with good expansion, we would really only still be looking at a convex combination of a mixed unitary quantum expander and a non-mixed unitary channel whose expansion properties are unknown. 
\end{remark}

\section{Coideals and spectral gap for quantum Schreier graphs}
\label{sec: Schreier}
In Section \ref{sec:edge-expanders}, we saw how irreducible finite-dimensional representations of property (T) discrete quantum groups can be used to construct quantum channels with uniform lower bounds on their spectral gap.  In this section, we outline another approach to constructing quantum expanders which is closer in spirit to the classical construction of expander graphs of \cite{Lub94, Mar73} obtained by considering finite Schreier coset graphs associated to discrete groups with property (T).  The key idea is to consider the appropriate quantum analogue of the (finite) quotient spaces $\Gamma/H$, where $\Gamma$ is a property (T) discrete group,  and $H \le \Gamma$ is a (finite-index) subgroup. For discrete quantum groups, this is captured by the notion of a {\it coideal}.


We use the notations of the previous sections for quantum groups.  In particular, $\cqG$ always denotes a compact quantum group and $\hat \cqG$ denotes its dual discrete quantum group.

\begin{definition}\label{def:coideal} Let $\hat \cqG$ be a discrete quantum group.  A {\it (left) coideal} of $\hat \cqG$ is a von Neumann subalgebra $M \subseteq \ell^\infty(\hat \cqG)$ such that 
\[\hat \Delta (M) \subseteq \ell^\infty(\hat \cqG) \overline{\otimes} M.\]
\end{definition}

It is well-known \cite{DKSS12} that when $\hat \cqG = \Gamma$ is a classical discrete group, coideals in $\Gamma$ are in one-to-one correspondence with the homogeneous spaces $\Gamma/H$, where $H \le \Gamma$ is a subgroup.  The identification is given by $M = \ell^\infty(\Gamma/H) \subseteq \ell^\infty(\Gamma)$, the von Neumann subalgebra of functions constant on the left cosets of $H$.   For general discrete quantum groups $\hat \cqG$, this analogy breaks down, and the theory of coideals turns out to be much richer than the study of quotients of $\hat \cqG$ by quantum subgroups.  In particular, while the notion of {\it quantum subgroup} $\hat H \subseteq \hat \cqG$ is well understood \cite{DKSS12}, and gives rise to a ``quotient-type'' coideal $M = \ell^\infty(\hat \cqG/\hat H)$, the, not all coideals arise from quantum subgroups.  See for example \cite{KK17,AK24},and references therein.       

Nonetheless, we still regard general coideals $M \subseteq \ell^\infty(\hat \cqG)$ as a kind of quantum homogeneous space. In particular, every coideal $M \subseteq \ell^\infty(\hat \cqG)$ admits a natural ergodic ``translation action'' of $\hat \cqG$ on $M$, which we now describe.

\begin{definition} \label{def:dqg-action}
A {\it (left) action} of a discrete quantum group $
\hat \cqG$ on a von Neumann algebra $M$ is a normal unital $\ast$-homomorphism \[\alpha: M \to \ell^\infty(\hat \cqG)\overline{\otimes}M\] satisfying the coassociativity condition \[(\hat \Delta \otimes \id)\alpha = (\id \otimes \alpha)\alpha.\]
We moreover call the action $\alpha$ {\it ergodic} if the fixed point algebra  $M^\alpha = \{x \in M : \alpha(x) = 1 \otimes x\}$ is trivial, i.e., \[M^\alpha = \mathbb C 1.\]
\end{definition}

The above definition generalises the usual notion of an action of a discrete group $\Gamma$ on $M$ by $\ast$-automorphisms.  Indeed, if $\Gamma$ is a classsical discrete group and $\Gamma \curvearrowright^\alpha M$ is an action given by \begin{align}\label{eqn:traditional-alpha}\Gamma \times M \owns (g,x) \mapsto \alpha_g(x) \in M,
\end{align} then one can define a map (still denoted by $\alpha$) $\alpha:M \to \ell^\infty(\Gamma) \overline{\otimes} M \cong \prod_{g \in \Gamma}^{\ell^\infty} M$ via
\begin{align} \label{eqn:nontraditional-alpha}
\alpha(x) = (\alpha_g(x))_{g \in \Gamma} \qquad (x \in  M).
\end{align}
Then one can easily verify that  \eqref{eqn:nontraditional-alpha} satisfies Definition \ref{def:dqg-action} if and only if \eqref{eqn:traditional-alpha} defines an action of $\Gamma$ on $M$ in the usual sense.

Returning to coideals $M \subseteq \ell^\infty(\hat \cqG)$ of a discrete quantum group $\hat \cqG$, the restriction of the comultiplication defines {\it the canonical (left) action}  $\alpha := \hat \Delta|_M$ of $\hat \cqG$ on $M$. Again in the classical case of a discrete group $\Gamma$, with $M = \ell^\infty(\Gamma/H)$ the coideal associated to a subgroup $H \le \Gamma$, the action $\alpha$ is just the canonical left translation action of $\Gamma$ on $\ell^\infty(\Gamma/H)$.  

Actions of quantum groups on von Neumann algebras are strongly linked to the theory of unitary representations.  Indeed, if $U \in M(C_0(\hat \cqG) \otimes B_0(\mathcal K))$ is a unitary representation of a discrete quantum group on a Hilbert space $\mathcal K$, then 
\[
\alpha: B(\mathcal K) \to \ell^\infty(\hat \cqG) \overline{\otimes} B(\mathcal K); \quad x \mapsto U^*(1 \otimes x)U \qquad (x \in B(\mathcal K))
\]
defines an action of $\hat \cqG$ on the von Neumann algebra $B(\mathcal K)$.  Conversely, a deep result of Vaes~\cite{Va01} shows that any action $\alpha$ of $\hat \cqG$ on a von Neumann algebra $M$ is {\it unitarily implemented} in the above sense. 
 That is, one can find a faithful representation $M \subseteq B(\mathcal K)$ and a unitary representation $U \in M(C_0(\hat \cqG) \otimes B_0(\mathcal K))$ such that 
 \begin{align} \label{eqn: unitary-implementation}
 \alpha(x) =  U^*(1 \otimes x)U \qquad (x \in M).
 \end{align}

\begin{remark}
In the above discussion, one could have equivalently defined unitarily implemented actions \eqref{eqn: unitary-implementation} via the formula $x \mapsto U(1 \otimes x)U^*$.   The convention \eqref{eqn: unitary-implementation} just ensures that $U$ is a representation of $
\hat \cqG$, as apposed to a representation of the opposite quantum group $\hat \cqG^{\text{op}}$.    
\end{remark}

When $M$ has admits a normal  faithful {\it $\alpha$-invariant} state $\varphi$ in the sense that $(\id \otimes \varphi) \circ \alpha = \varphi(\cdot) 1$, the 
unitary representation implementing $\alpha$ can be described somewhat explicitly \cite[Section 3]{DSV17}: Here one may take $\mathcal K = L^2(M,\varphi)$, the GNS Hilbert space of $\varphi$.  Then $U$ is given via its adjoint by 
\begin{align}\label{eqn:unitary-formula}
U^*(\xi \otimes \Lambda(x)) = \alpha(x)(\xi \otimes \Lambda(1)) \qquad (x \in M, \ \xi \in \ell^2(\hat \cqG) ),
\end{align}
where $\Lambda:M \to L^2(M,\varphi)$ is the GNS map and $\ell^2(\hat \cqG)$ is the GNS Hilbert space of the (right) Haar weight of $\hat \cqG$.

Let $M, \alpha, \varphi, U$ be as in the previous paragraph, and let $\psi \in \ell^1(\hat \cqG)$ be a normal state.  Then, just as in Section \ref{sec:non-tracial-gap}, we obtain a  normal UCP $\varphi$-preserving map
\begin{align}\label{eqn:Psi-from-action}
\Psi:M \to M; \qquad \Psi(x) = (\psi \otimes \id)\alpha(x) \qquad (x \in M).
\end{align}
Thanks to \eqref{eqn:unitary-formula}, the canonical extension of $\Psi$ to a bounded  map $\tilde\Psi \in B(L^2(M,\varphi))$ (densely defined by the formula $\tilde \Psi(\Lambda(x)) = \Lambda(\Psi(x))$) is given in terms of the implementing unitary representation $U$ by the formula 
\begin{align}\label{eqn:L2-implementation}
\tilde \Psi = (\psi \otimes \id) (U^*).
\end{align}
In particular, we have 
$(\tilde \Psi)^* = (\psi \otimes \id )U$, so $\tilde\Psi$ is self-adjoint when  $\psi = \psi \circ \hat S$. 

A special case of interest to us in the sequel in which actions are always guaranteed to admit faithful invariant states is when the underlying von Neumann algebra is finite-dimensional.  This result is surely folklore, but we provide a proof for self-containment.

\begin{lemma} \label{lem:fd-actions}
Let $\alpha$ be an action of a unimodular discrete quantum group $\hat \cqG$ on a finite-dimensional von Neumann algebra $M$.  Then $M$ admits a faithful $\alpha$-invariant state.   
\end{lemma}

\begin{proof}
Let $M \subseteq B(\mathcal K)$  and let $U \in M(C_0(\hat \cqG) \otimes B_0(\mathcal K))$ be an implementing unitary for the action $\alpha$.   We make use of the {\it quantum Bohr compactification} $b\hat\cqG$ of $\hat \cqG$ \cite{So05}.  By definition, $b\hat \cqG$ is the compact quantum group with $C(b\hat \cqG)$ given by the C*-subalgebra of $\ell^\infty(\hat\cqG)$ generated by the matrix coefficients of the so-called  finite dimensional {\it admissible} unitary representations of $\hat \cqG$, with comultiplication given by the restriction of $\hat \Delta$.  Here, admissible means that the transpose of $U$, when viewed as a matrix, is invertible.  Since $\hat \cqG$ is unimodular, every  finite-dimensional representation $U$ is admissible \cite[Theorem 2.1]{Vi17}, and since the action $\alpha$ is implemented by $U$, which is a representation of $b \hat \cqG$,  it can also be regarded as an action of $b\hat \cqG$ on $M$.  The claim now follows from the fact that any action of a compact quantum group on a finite-dimensional von Neumann algebra $M$ admits a faithful invariant state $\psi$. Indeed, one can just take $\psi = (h \otimes \psi_0)\circ \alpha$, where $h$ is the Haar state of $b\hat \cqG$ and $\psi_0$ is any faithful state  on $M$.            
\end{proof}

\subsection{Quantum Cayley graphs and Schreier graphs}

We now introduce quantum Schreier graphs.  A quantum Schreier graph is a natural extension of the notion of a quantum Cayley graph, as described in \cite{Was23}, and we start off by discussing these.  We note that a version of quantum Cayley graphs was also developed earlier by Vergnioux \cite{Ver05}, using a rather different-looking, but ultimately equivalent language.  We will use the notation and terminology of  \cite{Was23} below.

Let $\hat \cqG$ be a unimodular discrete quantum group. For $x \in C_{00}(\hat \cqG)$, we consider the normal functional $\omega_x := \hat h(\cdot x) \in \ell^1(\hat \cqG)$.  Given $x,y \in C_{00}(\hat \cqG)$, there exists a unique $z \in C_{00}(\hat \cqG)$ such that $\omega_z = \omega_x \star \omega_y$.  This defines an associative product on $C_{00}(\hat \cqG)$, also called convolution, and denoted $z = x \star y$ \cite{Ver05,Ve07}.  In fact, when $\hat G$ is unimodular, we have \cite[Lemma 4.3]{BR17}
\begin{align}\label{eqn:conv-formula}
x\star y = (\omega_x\circ \hat S \otimes \id) \hat \Delta(y).    
\end{align}
In particular, the left convolution map $y \mapsto x\star y$ is nothing other than  the normal completely bounded map $\Psi_{\omega_x}:\ell^\infty(\hat \cqG) \to \ell^\infty(\hat \cqG)$ 
associated to the action $\alpha = \hat\Delta$ of $\hat \cqG$ on  $\ell^\infty(\hat \cqG)$.  This map is completely positive when $\omega_x \circ \hat S$ is a positive functional. 

To define a quantum Cayley graph over $\hat \cqG$, we fix a finite rank projection $p \in C_{00}(\hat \cqG)$ such that $\hat S(p) = p$, $\hat \epsilon(p) = 0$ and $\bigvee_{n \in \mathbb N} p^{\star n} = 1_{\ell^\infty(\hat \cqG)}$.  Define the normal completely positive map $A:\ell^\infty(\hat \cqG) \to \ell^\infty(\hat \cqG)$ by 
\[Ax = p\star x \qquad (x \in \ell^\infty(G)).\]

\begin{definition}[\cite{Was23}] \label{def:qCayley}
The pair $(\ell^\infty(\hat \cqG), A)$ is called a {\it quantum Cayley graph over $\hat \cqG$}.    
\end{definition}  As explained in \cite{Was23}, when $\Gamma = \hat \cqG$ is a classical discrete group, the above construction yields $p = \chi_E$, the characteristic function of a symmetric generating set $e \notin E = E^{-1}$ of $\Gamma$, and $Af(g) = (\chi_E\star f)(g) = \sum_{s\in E}f(sg)$ is the usual adjacency operator of the Cayley graph $\mathcal C(\Gamma, E)$.  We remark that quantum Cayley graphs are always {\it $d$-regular quantum graphs} with $d = \hat h(p)$, since 
\[
A1 = \hat h(p)1.
\]
We also note that
\[
\hat h(Ax) = \hat h(p\star x) = \hat h(p) \hat h(x),
\]
for all $x \in C_{00}(\hat \cqG)$.  As a consequence, the normalised adjacency operator is $d^{-1}A$ is a bistochastic quantum channel relative to the tracial Haar weight $\hat h$.

Let $M \subseteq \ell^\infty(\hat \cqG)$ be a coideal.  It is clear from the definitions that the quantum adjacecncy matrix $A$ associated to the quantum Cayley graph in Definition \ref{def:qCayley} leaves $M$ invariant. In particular, in the case of classical discrete groups, restriction $A|_{M}$ is the natural analogue of the adjacency matrix of the associated Schreier coset graph. 
 This leads us to the following definition.

 \begin{definition}\label{qSchreier}
 Let  $(\ell^\infty(\hat \cqG), A)$ be a ($d$-regular) quantum Cayley graph over $\hat \cqG$.  For any coideal $M \subseteq \ell^\infty(\hat \cqG)$, the pair $(M, A|_{M})$ is called the {\it ($d$-regular) quantum Schreier graph} associated to the generating projection $p$.   
 \end{definition}

\subsection{Spectral gap for quantum Schreier graphs}

We now come to the main result of the section, which asserts that quantum Schrier graphs associated to property (T) quantum groups give rise to bounded degree expanders.  

Let $\hat \cqG$ be a discrete quantum group with property (T). Let $(E, \epsilon)$ be a fixed Kazhdan pair for $\hat \cqG$ with $E = \bar E$.  Let $p_E \in C_{00}(\hat \cqG)$ be the central support of $E$, let $M \subseteq \ell^\infty(\hat \cqG)$ be a coideal, and let $(M,A|_{M})$ be the quantum Schreier graph associated to the generating projection $p_E$.  Finally, assume that $M$ admits a normal invariant state $\varphi$.  In this case (thanks to the assumption $\hat S(p) = p)$, we have that $A|_M$ extends to a self-adjoint map $\widetilde{A|_M}:L^2(M,\varphi) \to L^2(M,\varphi)$. By Lemma \ref{lem:fd-actions}, this is always true if $\dim M < \infty$.  We denote by $\lambda_2(M,E) \le 1$ the second largest eigenvalue of $d^{-1}\widetilde{A|_M}$, where $d = \hat h(p_E) = \sum_{s \in E} (\dim s)^2$ is the degreee.  

\begin{theorem}\label{thm:spectral-gap-schreier}
Let $\hat \cqG$ and $(E,\epsilon)$ be as above.  For any finite dimensional quantum Schreier graph $(M,A|_M)$ be  associated to $p_E$, we have 
\[
\lambda_2(M,E)  \le 1-\frac{\lambda\epsilon^2}{2}
\]
where $\lambda= \min_{s \in E} \frac{\dim \mathcal H_s}{d}$.
\end{theorem}

\begin{proof}
The proof is similar in spirit to the proofs of Theorems \ref{thm:lambda2} and \ref{thm:lambda2-nontracial}.  Let $\varphi$ be a faithful invariant state on $M$ with respect to the action $\alpha = \hat \Delta|_M$.  Let $U \in M(C_0(\hat \cqG) \otimes B_0(L^2(M,\varphi)))$ be the unitary implementation of $\alpha$.  Then by \eqref{eqn:L2-implementation}, we need to compute the second largest eigenvalue $\lambda_2(\tilde \Psi)$ of the self-adjoint map 
\[\tilde \Psi = (\psi \otimes \id)(U^*) = (\psi \otimes \id)(U),\]
where we are working with the $\hat S$-invariant state $\psi = d^{-1}\hat h(\cdot p_E) \in \ell^1(\hat \cqG)$.  Note that since the action $\alpha = \hat \Delta|_M$ is ergodic, it follows from equation \eqref{eqn:unitary-formula} (describing $U$) that the subspace $L^2(M,\varphi)_0$ of $U$-invariant vectors is exactly $\Lambda(M^\alpha) = \mathbb C \Lambda(1)$.

We thus need to establish the inequality
\begin{align}\label{eqn:suchfun}\inf_{\xi \in \Lambda(1)^\perp, \ \|\xi\|=1}\latRe\braket{\xi}{(\psi \otimes \id )(U)}{\xi}\leq 1-\frac{\lambda\epsilon^2}{2}\end{align}
Recalling now the notation $\mathcal H_E = \bigoplus_{s \in E}\mathcal H_s$, and the canonical embedding $p_E\ell^\infty(\hat \cqG)p_E \subseteq B(\mathcal H_E)$ from Section \ref{sec:edge-expanders}, we may canonically extend $\psi$ to a (non-tracial state) on $B(\mathcal H_E)$ with eigenvalues given (without multiplicity) by $\{\frac{\dim \mathcal H_s}{d}\}_{s \in E}$.  (The extended $\psi$ is nothing more than $\psi \circ \mathbb E$, where $\mathbb E:B(\mathcal H_E) \to p_E\ell^\infty(\hat \cqG)p_E$ is the trace-preserving conditional expectation.)   With all this setup in place, we now see that the desired inequality \eqref{eqn:suchfun} is just a special case of the inequality \eqref{eqn:non-tracial-basic} obtained in the proof of Theorem \ref{thm:lambda2-nontracial}.
\end{proof}

\section{Discussion} \label{sec:discussion}
In this work, we provide two general approaches to construct quantum expanders from discrete quantum groups with property (T).  The first approach (Section \ref{sec:edge-expanders})  takes as input a property (T) discrete quantum group together with a family of irreducible finite-dimensional representations with growing dimension, while the second approach (Section \ref{sec: Schreier}) takes as input a property (T) discrete quantum group together with a family of finite-dimensional coideals with unbounded dimension. 
Ultimately, what is needed for either method to produce examples is a good understanding of the finite-dimensional representation theory of a given property (T) discrete quantum group.  However, at this time, it appears that there is a dearth of non-classical examples where the finite dimensional representation theory is well understood.    In this section we will briefly discuss how the example we do have (quantum expanders coming from bicrossed products) should be viewed and what steps need to be taken to find more examples.

\subsection{Quantum expanders coming from bicrossed products}
The most natural quantum expanders coming from bicrossed products are convex combinations of quantum channels of the form given in (\ref{eq:tracial-bicrossed-product-channel}), 
\begin{equation*}
    \Phi(\rho)=\frac{1}{|\gamma \cdot G|}\sum_{r,s\in \gamma\cdot G}\pi(u_r)\pi(\omega(v_{r,s}))\rho \pi(\omega(v_{r,s}))^*\pi(u_r)^*.
\end{equation*}
However, taking the same convex combination of quantum channels of the form
\begin{equation*}
    \Phi'(\rho)=\frac{1}{|\gamma \cdot G|}\sum_{r,s\in \gamma\cdot G}\pi(u_r)\rho \pi(u_r)^*,
\end{equation*}
which are quantum channels corresponding to the classical property (T) discrete group $\Gamma$, already gives a quantum expander. If one composes two unital trace-preserving completely positive maps $\Phi$ and $\Psi$, then a spectral gap of either $\Phi$ or $\Psi$ ensures a spectral gap of at least the same size in $\Phi\circ\Psi$. How should we then view the quantum expanders coming from bicrossed products? The starting point is the quantum expander coming from the property (T) discrete group. This channel is a mixed unitary channel, and in the quantum expander coming from the bicrossed product each unitary conjugation is precomposed by a \textit{different} quantum channel. In general, doing this would no longer give you a quantum expander, but our results show that precomposing by these particular quantum channels preserves the spectral gap.

\subsection{Representation theory of other property (T) quantum groups}
As mentioned previously, there are several recent geometric constructions which give rise to examples of property (T) discrete quantum groups that are different from the bicrossed product construction.  These include the discrete quantum groups associated to triangle presentations \cite{VaVa19}, and the quantisations of discrete groups \cite{RV24}.  As pointed out in \cite{RV24}, there is also the quantum automorphism group $\text{Aut}^+(HS)$ of the Higman-Sims graph $HS$, whose dual is an infinite discrete quantum group with  property (T) .  All of these examples provide new potential candidates for the application of our machinery, as soon as their representation theory is understood.  In particular, we think it is a very interesting problem to better  understand the structure of the quantum automorphism group of the Higman-Sims graph $HS$, $\text{Aut}^+(HS)$, and its discrete dual.  For example, does the Hopf*-algebra $\mathcal O(\text{Aut}^+(HS))$ have many non-abelian finite-dimensional representations? 
 Is it residually finite dimensional? 

\subsection{Finding examples of quantum Schreier graphs}

The approach to quantum expanders via quantum Schreier graphs is very promising  because, on the one hand it gives a direct generalization of the classical construction of quantum expanders of Margulis \cite{Mar73}, while on the other hand it yields potentially new types of quantum expanders acting on general finite-dimensional von Neumann algebras (and not just full matrix algebras, as is standard in the literature).  The problem here again is to find non-trivial examples to apply these tools. We leave open the problem of finding non-classical discrete quantum groups with property (T) with large finite-dimensional coideals.  Even for the property (T) quantum groups coming from bicrossed products,  we do not know much about the structure of their coideals.  

\subsection*{Acknowledgements}
MB was partially supported by an NSERC discovery grant.  EC was partially supported by an NSERC Canada Graduate Scholarship. MV was supported by the NWO Vidi grant VI.Vidi.192.018 `Non-commutative harmonic analysis and rigidity of operator
algebras'. The authors would like to thank the Isaac Newton Institute (INI) for Mathematical Sciences, Cambridge, for support and hospitality during the programme Quantum information, quantum groups and operator algebras, where work on this paper was undertaken. The work undertaken at INI was supported by EPSRC grant EP/Z000580/1.  The authors also thank Archishna Bhattacharyya and Benjamin Anderson-Sackaney  for helpful discussions on fixed points of channels and coideals. 

\bibliographystyle{alphaarxiv.bst}
\bibliography{refs.bib}

\end{document}